\documentclass[11pt]{amsart}
\usepackage{amssymb}

\usepackage{amssymb,amsfonts,bm,amsmath}
\usepackage{hyperref}
\usepackage[all]{xy}

\theoremstyle{plain}
\newtheorem{thm}{Theorem}[section]
\newtheorem{theorem}[thm]{Theorem}
\newtheorem{lemma}[thm]{Lemma}

\newtheorem{corollary*}[thm]{Corollary*}
\newtheorem{proposition}[thm]{Proposition}
\newtheorem{proposition*}[thm]{Proposition*}
\newtheorem{conjecture}[thm]{Conjecture}
\theoremstyle{definition}

\newtheorem{definition}[thm]{Definition}

\newtheorem{defn-thm}[thm]{Definition-Theorem}

\newcommand{\sM}{{\mathcal M}}
\def\rank{\operatorname{rank}}

\allowdisplaybreaks

\title[Descendent integrals and tautological rings]{Descendent integrals and tautological rings of\\ moduli spaces of curves}

\author{Kefeng Liu}
\dedicatory{Dedicated to Professor Shing-Tung Yau on the occasion of
his 60th birthday}
        \address{Center of Mathematical Sciences, Zhejiang University, Hangzhou, Zhejiang 310027, China;
                Department of Mathematics,University of California at Los Angeles,
                Los Angeles, CA 90095-1555, USA}
        \email{liu@math.ucla.edu, liu@cms.zju.edu.cn}
        \author{Hao Xu}
        \address{Center of Mathematical Sciences, Zhejiang University, Hangzhou, Zhejiang 310027, China}
        \email{haoxu@cms.zju.edu.cn}


\begin{document}

\begin{abstract}
The main objective of this paper is to give a summary of our recent
work on recursion formulae for intersection numbers on moduli spaces
of curves and their applications. We also present a conjectural
relation between tautological rings and the mock theta function.
\end{abstract}

\maketitle

\tableofcontents



\vskip 30pt
\section{Introduction}

The moduli space of curves is a very important subject of study in
algebraic geometry and mathematical physics. Mathematicians
realized long ago that the study of a single curve is not sufficient
even to understand the curve itself; the most fascinating aspect of
the story is that properties of curves can be better understood when
comparing general members in a natural family of curves in question.

The study of moduli spaces of curves started with Riemann. Although
its construction was consolidated only in the late 1960's in the hands of
Mumford, Deligne, Gieseker, etc., the moduli space of curves was
already implicit in the work of the classical Italian school of
algebraic geometry. It has never lost its value since its creation;
we see much beautiful mathematics flourishing around it in the past
years.

In this paper, we will only touch on one aspect of the moduli space of
curves, namely the intersection theory.
There is a large amount of literature on the intersection theory of
moduli spaces of curves, a nice account can be found in Vakil's survey paper
\cite{Va1}.

Since Mumford's pioneering work \cite{Mu}, we know that certain
geometrically natural cohomology classes on moduli spaces of curves are of primary interest.

A ground-breaking achievement in the early 1990's is the celebrated
Witten conjecture, giving a surprising connection of intersection theory
on moduli spaces of curves to the realm of integrable systems.
Shortly afterwards, Kontsevich gave a remarkable proof by expressing integrals of $\psi$ classes
as a new type of matrix integrals. Moreover, the Witten-Kontsevich theorem provides
an effective recursive way to compute
integrals of $\psi$ classes, or descendent integrals. On the other
hand, Gromov-Witten theory, also dubbed modern enumerative
geometry, revolutionized the study of Hodge integrals on moduli spaces of
curves, namely integrals of mixed $\psi$ and $\lambda$ classes.

Another important development in the 1990's is Faber's remarkable
conjectures on the tautological rings of moduli spaces of curves.

In the meantime, string duality has produced many conjectures about
the moduli spaces of stable curves and stable maps, such as the
mirror formula and the Mari\~no-Vafa formula. The interactions of
string theory and moduli spaces has been one of the  most exciting
research fields in mathematics for the past several years. For more
discussion, see the survey article \cite{Liu}.

In this paper, we first review the Witten-Kontsevich theorem, which
is the starting point of many results discussed here. After that, we
will give a survey of our work on explicit effective recursion
formulae for computing descendent integrals, higher Weil-Petersson
volumes of moduli spaces of curves and Witten's $r$-spin
intersection numbers. Most of these formulae provide feasible
algorithms for implementation on computers.

We also describe our proof of the Faber intersection number
conjecture using a recursive formula of $n$-point functions and discuss possible relations between tautological
rings and mock theta functions.

\

\noindent{\bf Acknowledgements.} The authors would like to thank
Professors Carel Faber, Sergei Lando, Jun Li, Chiu-Chu Melissa Liu,
Xiaobo Liu, Ravi Vakil and Jian Zhou for helpful communications
related to results in this paper. We thank Motohico Mulase, Brad
Safnuk and Bertrand Eynard for helpful conversations during the
International Academic Symposium on Riemann Surfaces, Harmonic Maps
and Visualization at Osaka City University in December, 2008. The
second author also thanks Kathrin Bringmann for explaining the
Kronecker symbol to him. Part of this paper was prepared during the
second author's visit to the Institute of Mathematics, Chinese Academy
of Sciences in June, 2008. He thanks Professor Xiangyu Zhou for
hospitality. The second author also wishes to thank Professor S.-T.
Yau for constant encouragement and helpful suggestions.

We thank Arthur Greenspoon of AMS for his careful proofreading,
which improved our exposition greatly. Any remaining errors are the responsibility of the authors.
We thank Professor Lizhen Ji for the tremendous help in editing and refereeing of this paper.

\vskip 30pt
\section{Intersection numbers and the Witten-Kontsevich theorem}
 Denote by $\overline{\sM}_{g,n}$ the moduli space of stable
$n$-pointed genus $g$ complex algebraic curves. We have the morphism
that forgets the last marked point
$$
\pi: \overline{\sM}_{g,n+1}\longrightarrow \overline{\sM}_{g,n}.
$$
Denote by $\sigma_1,\dots,\sigma_n$ the canonical sections of $\pi$. Let $\omega_{\pi}$ be the relative
dualizing sheaf; we have the following tautological classes on
moduli spaces of curves.
\begin{align*}
\psi_i&=c_1(\sigma_i^*(\omega_{\pi}))\\
\kappa_i&=\pi_*(\psi_{n+1}^{i+1})\\
\lambda_k&=c_k(\mathbb E),\quad 1\leq k\leq g,
\end{align*}
where $\mathbb E=\pi_*(\omega_{\pi})$ is the Hodge bundle.

Intuitively, $\psi_i$ is the first Chern class of the line bundle
corresponding to the cotangent space of the universal curve at the
$i$-th marked point and the fiber of $\mathbb E$ is the space of
holomorphic one-forms on the algebraic curve.

We use Witten's notation
$$\langle\tau_{d_1}\cdots\tau_{d_n}\kappa_{a_1}\cdots\kappa_{a_m}\mid\lambda_{1}^{k_{1}}
    \cdots\lambda_{g}^{k_{g}}\rangle:= \int_{\overline{\mathcal{M}}_{g,n}}\psi_{1}^{d_{1}}\cdots\psi_{n}^{d_{n}}\kappa_{a_1}\cdots\kappa_{a_m}\lambda_{1}^{k_{1}}
    \cdots\lambda_{g}^{k_{g}}.$$
These intersection numbers are called the Hodge integrals. They are
rational numbers because the moduli spaces of curves are orbifolds
(with quotient singularities) except in genus zero. Their degrees
should add up to $\dim \overline{\mathcal{M}}_{g,n}= 3g-3+n$.

Intersection numbers of pure $\psi$ classes
$\langle\tau_{d_1}\cdots\tau_{d_n}\rangle$ are often called
descendent integrals. Intersection numbers of pure $\kappa$ classes
$\langle\kappa_{a_1}\cdots\kappa_{a_m}\rangle$ are called higher
Weil-Petersson volumes.

 The classes $\kappa_i$ were
first introduced by Mumford \cite{Mu} on $\overline{\sM}_g$; their
generalization to $\overline{\sM}_{g,n}$ here is due to
Arbarello-Cornalba \cite{Ar-Co}.

Mumford's 1983 paper \cite{Mu} initiated the systematic study of the
intersection theory of the moduli space of curves. In particular,
Mumford computed the Chow ring of $\overline{\sM}_{2}$,
$$A^*(\overline{\sM}_{2})\cong\mathbb[\lambda,\delta]/
(5\lambda^3-\lambda^2\delta,110\lambda^2-21\lambda\delta+\delta^2),$$
where $\lambda=\lambda_1$ and $\delta=\delta_0+\delta_1$ is the full boundary divisor.

The Chow rings of $\overline{\sM}_{g}$ for $g\leq5$ have been studied by
Faber \cite{FaChow} and Izadi \cite{Iz}.

\subsection{Witten-Kontsevich theorem}
In 1990, Witten \cite{Wi} made a striking conjecture (first proved
by Kontsevich \cite{Ko}) that the generating function of $\psi$
class intersection numbers is governed by the KdV hierarchy. Now
Witten's conjecture has many different proofs \cite{CLL, Ka, KL,
KiL, Mir2, OP}.

Roughly speaking, Witten's motivation comes from the two seemingly
unrelated mathematical models that describe the physical theory of
two-dimensional gravity. One is the counting of triangulations of
surfaces, which is related to matrix models and the other is the
intersection theory of $\overline{\sM}_{g,n}$. The partition
function of the first model is known to obey the KdV hierarchy.

Kontsevich's remarkable proof uses a combinatorial description of
moduli spaces and Feynman diagram techniques. A very readable
exposition can be found in \cite{Lo1}.

Witten's conjecture revolutionized the intersection theory of moduli
spaces of curves and motivated a surge of subsequent developments:
Gromov-Witten theory, Faber's conjecture \cite{Fa2} and the Virasoro
conjecture of Eguchi-Hori-Xiong-Katz \cite{EHX}. Below we follow
Witten's nice exposition \cite{Wi}.

The KdV hierarchy is the following hierarchy of differential
equations for $n\geq 1$,
\begin{equation}\label{kdv}
\frac{\partial U}{\partial t_n}=\frac{\partial R_{n+1}}{\partial
t_0},
\end{equation}
 where $R_n$ are Gelfand-Dikii differential polynomials
in $U,\partial U/\partial t_0,\partial^2 U/\partial t_0^2,\dots$,
defined recursively by
\begin{equation}\label{kdvrec}
R_1=U,\qquad \frac{\partial R_{n+1}}{\partial
t_0}=\frac{1}{2n+1}\left(\frac{\partial U}{\partial
t_0}R_n+2U\frac{\partial R_n}{\partial t_0}+
\frac{1}{4}\frac{\partial^3 R_n}{\partial t_0^3}\right).
\end{equation}

 It is easy to see that
$$R_2=\frac{1}{2}U^2+\frac{1}{12}\frac{\partial^2U}{\partial
t_0^2},$$
$$R_3=\frac{1}{6}U^3+\frac{U}{12}\frac{\partial^3U}{\partial
t_0^3}+\frac{1}{24}(\frac{\partial U}{\partial
t_0})^2+\frac{1}{240}\frac{\partial^4U}{\partial t_0^4},$$
$$\vdots$$

The Witten-Kontsevich theorem states that the generating function
\begin{equation}\label{gen}
F(t_0, t_1, \ldots)= \sum_{g} \sum_{\bold n}
\langle\prod_{i=0}^\infty \tau_{i}^{n_i}\rangle_{g}
\prod_{i=0}^\infty \frac{t_i^{n_i} }{n_i!}
\end{equation}
 is a $\tau$-function for
the KdV hierarchy, i.e. $U=\partial^2 F/\partial t_0^2$ obeys all
equations in the KdV hierarchy. The first equation in the KdV
hierarchy is the classical KdV equation
\begin{equation}\label{firstkdv}
\frac{\partial U}{\partial t_1}=U\frac{\partial U}{\partial
t_0}+\frac{1}{12}\frac{\partial^3 U}{\partial t_0^3}.
\end{equation}

In addition, $F$ obeys the string equation
\begin{equation}\label{str}
\frac{\partial F}{\partial t_0}=\frac{t_0^2}{2}+\sum_{i=0}^\infty
t_{i+1}\frac{\partial F}{\partial t_i}
\end{equation}
and the dilaton equation
\begin{equation}\label{dil}
\frac{\partial F}{\partial t_1}=\frac{1}{24}+\sum_{i=0}^{\infty}\frac{2i+1}{3}t_i\frac{\partial F}{\partial t_i}.
\end{equation}

The string and dilaton equations can be proved directly in algebraic geometry,
e.g. see \cite{LZ}. Witten introduced the following notation for the derivatives of $F$,
$$\langle\langle\tau_{d_1}\cdots\tau_{d_n}\rangle\rangle:=\frac{\partial^nF}{\partial
t_{d_1}\cdots\partial t_{d_n}}.$$
The left-hand side of \eqref{kdv} is the same as
\begin{equation*}
\frac{\partial}{\partial
t_0}\langle\langle\tau_n\tau_0\rangle\rangle.
\end{equation*}
By the string equation, it is clear that upon integrating both sides of
\eqref{kdv} once in $t_0$ and using the recursion relation
\eqref{kdvrec}, we get
\begin{equation}\label{kdvcoeff}
\langle\langle\tau_n\tau_0\tau_0\rangle\rangle=\frac{1}{2n+1}\langle\langle\tau_{n-1}\tau_0\rangle\rangle\langle\langle\tau_0^3\rangle\rangle
+2\langle\langle\tau_{n-1}\tau_0^2\rangle\rangle\langle\langle\tau_0^2\rangle\rangle+\frac{1}{4}\langle\langle\tau_{n-1}\tau_0^4\rangle\rangle.
\end{equation}

The above discussion may be summarized as the following:
\begin{proposition}
Let $F$ be the generating function \eqref{gen} and $U=\partial^2
F/\partial t_0^2$. Then we have the following equivalent statements
of the Witten-Kontsevich theorem:
\begin{itemize}
\item[i)] U satisfies the KdV hierarchy \eqref{kdv} and the string
equation \eqref{str};
\item[ii)] U satisfies the first KdV equation \eqref{firstkdv}, the string
equation and the dilaton equation;
\item[iii)] F satisfies the recursion formula \eqref{kdvcoeff} and the string
equation.
\end{itemize}
Moreover, the Witten-Kontsevich theorem uniquely determines $F$.
\end{proposition}
\begin{proof}
The only nontrivial part is that (ii) implies (i), which is proved in \cite{LX7} (Corollary 2.4).
\end{proof}

\subsection{Virasoro constraints} The Witten-Kontsevich theorem has an important reformulation
due to Dijkgraaf, Verlinde, and Verlinde \cite{DVV}, in terms of the
Virasoro constraints.

Define a family of differential operators $L_k$ for $k\geq-1$ by
\begin{multline} \label{vir}
    L_k = -\frac{1}{2}  (2k+3)!!
        \frac{\partial }{\partial t_{k+1} }
        + \frac{1}{2} \sum_{j=0}^{\infty} \frac{(2(j+k)+1)!! }{(2j-1)!! } t_j
        \frac{\partial }{\partial t_{j+k} } \\
    + \frac{1}{4} \sum_{d_1 + d_2 = k-1}
        (2d_1 + 1)!! (2d_2 + 1)!! \frac{\partial^2 }{\partial t_{d_1}
        \partial t_{d_2}}
        + \frac{\delta_{k,-1}t_0^2}{4} + \frac{\delta_{k,0}
        }{48},
 \end{multline}
It is straightforward to verify that these operators satisfy the
Virasoro relations
$$[L_n,L_m]=(n-m)V_{n+m}.$$

Dijkgraaf, Verlinde, and Verlinde \cite{DVV} have proved that the
KdV form of Witten's conjecture is equivalent to the following
Virasoro constraints. An elegant exposition of the proof can be
found in \cite{Ge}.
\begin{proposition} (DVV formula) Let $F$ be the generating function of descendent integrals defined in \eqref{gen}.
 We have $L_k(\exp F)=0$ for $k\geq -1$. More explicitly,
\begin{multline}\label{dvv}
\langle\tau_{k+1}\tau_{d_1}\cdots\tau_{d_n}\rangle_g=\frac{1}{(2k+3)!!}\left[\sum_{j=1}^n
\frac{(2k+2d_j+1)!!}{(2d_j-1)!!}\langle\tau_{d_1}\cdots
\tau_{d_{j}+k}\cdots\tau_{d_n}\rangle_g\right.\\
+\frac{1}{2}\sum_{r+s=k-1}
(2r+1)!!(2s+1)!!\langle\tau_r\tau_s\tau_{d_1}\cdots\tau_{d_n}\rangle_{g-1}\\
\left.+\frac{1}{2}\sum_{r+s=k-1} (2r+1)!!(2s+1)!!
\sum_{\underline{n}=I\coprod J}\langle\tau_r\prod_{i\in
I}\tau_{d_i}\rangle_{g'}\langle\tau_s\prod_{i\in
J}\tau_{d_i}\rangle_{g-g'}\right].
\end{multline}
\end{proposition}

\vskip 30pt
\section{The $n$-point function}

\begin{definition} In \cite{Fa2},
the following generating function,
$$F(x_1,\cdots,x_n)=\sum_{g=0}^{\infty}\sum_{\sum d_i=3g-3+n}\langle\tau_{d_1}\cdots\tau_{d_n}\rangle_g\prod_{i=1}^n x_i^{d_i}$$
is called the $n$-point function.
\end{definition}

The coefficients of $n$-point functions encode all information of
intersection numbers of $\psi$ classes. Okounkov \cite{Ok2} obtained
a beautiful expression of the $n$-point functions using
$n$-dimensional error-function-type integrals. However, it is very
difficult to extract coefficients from Okounkov's analytic formula.

Br\'ezin and Hikami \cite{BH} apply correlation functions of GUE
ensemble to uncover explicit formulae of $n$-point functions.

\subsection{A recursive formula of $n$-point functions}

Consider the following ``normalized'' $n$-point function
$$G(x_1,\dots,x_n)=\exp\left(\frac{-\sum_{j=1}^n
x_j^3}{24}\right) F(x_1,\dots,x_n).$$

The one-point function $G(x)=\frac{1}{x^2}$ is due to Witten; we
have also Dijkgraaf's two-point function
$$G(x,y)=\frac{1}{x+y}\sum_{k\geq0}\frac{k!}{(2k+1)!}\left(\frac{1}{2}xy(x+y)\right)^k$$
and Zagier's three-point function \cite{Za} which we learned from
Prof. C. Faber,
\begin{equation*}
G(x,y,z)=\sum_{r,s\geq
0}\frac{r!S_r(x,y,z)}{4^r(2r+1)!!\cdot2}\cdot\frac{\Delta^s}{8^s(r+s+1)!},
\end{equation*}
where $S_r(x,y,z)$ and $\Delta$ are the homogeneous symmetric
polynomials defined by
\begin{align*}
S_r(x,y,z)&=\frac{(xy)^r(x+y)^{r+1}+(yz)^r(y+z)^{r+1}+(zx)^r(z+x)^{r+1}}{x+y+z}\in\mathbb Z[x,y,z],\\
\Delta(x,y,z)&=(x+y)(y+z)(z+x)=\frac{(x+y+z)^3}{3}-\frac{x^3+y^3+z^3}{3}.
\end{align*}

The two- and three-point functions were found in the early 1990's.
These explicit form of two- and three-point functions played a crucial role in
Faber's pioneering work on tautological rings \cite{Fa2}.
Since then it has been a prominent open problem
to find explicit formulae of $n$-point functions, closed or recursive. Faber's work \cite{Fa2}
indicated clearly that this is probably the first step toward a proof of his intersection
number conjecture.

By
solving the differential equation coming from Witten's KdV
coefficient equation \eqref{kdvcoeff}, we get a recursive formula
for normalized $n$-point functions, generalizing Dijkgraaf and
Zagier's formulae.

\begin{theorem}\cite{LX7}\label{npt} For $n\geq2$,
\begin{equation*}
G(x_1,\dots,x_n)=\sum_{r,s\geq0}\frac{(2r+n-3)!!}{4^s(2r+2s+n-1)!!}P_r(x_1,\dots,x_n)\Delta(x_1,\dots,x_n)^s,
\end{equation*}
where $P_r$ and $\Delta$ are homogeneous symmetric polynomials
defined by
\begin{align*}
\Delta(x_1,\dots,x_n)&=\frac{(\sum_{j=1}^nx_j)^3-\sum_{j=1}^nx_j^3}{3},\nonumber\\
P_r(x_1,\dots,x_n)&=\left(\frac{1}{2\sum_{j=1}^nx_j}\sum_{\underline{n}=I\coprod
J}(\sum_{i\in I}x_i)^2(\sum_{i\in J}x_i)^2G(x_I)
G(x_J)\right)_{3r+n-3}\nonumber\\
&=\frac{1}{2\sum_{j=1}^nx_j}\sum_{\underline{n}=I\coprod
J}(\sum_{i\in I}x_i)^2(\sum_{i\in J}x_i)^2\sum_{r'=0}^r
G_{r'}(x_I)G_{r-r'}(x_J),
\end{align*}
where $I,J\ne\emptyset$, $\underline{n}=\{1,2,\ldots,n\}$ and
$G_g(x_I)$ denotes the degree $3g+|I|-3$ homogeneous component of
the normalized $|I|$-point function $G(x_{k_1},\dots,x_{k_{|I|}})$,
where $k_j\in I$.
\end{theorem}

Dijkgraaf and Zagier's formulae gave us much
inspiration in writing down the general pattern of the above recursive formula.
On the other hand, it took us great effort to get the correct coefficients on the right hand side.

\begin{proposition}\cite{LX7}\label{npt2} The recursion relation in Theorem \ref{npt} is equivalent to either one of the following
statements.
\begin{enumerate}
\item[i)] The normalized $n$-point functions satisfy the following recursion relation
$$G_g(x_1,\dots,x_n)=
\frac{1}{(2g+n-1)}P_g(x_1,\dots,x_n)+\frac{\Delta(x_1,\dots,x_n)}{4(2g+n-1)}G_{g-1}(x_1,\dots,x_n).$$

\item[ii)] The $n$-point functions $F_g(x_1,\dots,x_n)$ satisfy the following recursion relation
\begin{multline*}
 (2g+n-1)\left(\sum_{i=1}^n x_i\right)F_g(x_1,\dots,x_n)=\frac{1}{12}\left(\sum_{i=1}^n x_i\right)^4F_{g-1}(x_1,\dots,x_n)\\
+\frac{1}{2}\sum_{g'=0}^g\sum_{\underline{n}=I\coprod J} \left(\sum_{i\in I}
x_i\right)^2\left(\sum_{i\in J} x_i\right)^2 F_{g'}(x_I)F_{g-g'}(x_J).
\end{multline*}
\end{enumerate}
\end{proposition}

After we used Theorem \ref{npt} to prove Proposition
\ref{npt2}(ii), we realized that the latter has already been embodied
in the Witten-Kontsevich theorem. Integrating the KdV
equation \eqref{firstkdv} once in $t_0$, we get
\begin{equation}\label{kdv2}
\frac{\partial^2 F}{\partial t_1\partial
t_0}=\frac{1}{2}\left(\frac{\partial F}{\partial t_0^2}\right)^2+
\frac{1}{12}\frac{\partial^4 F}{\partial t_0^4}+H(t_1,t_2,\dots),
\end{equation}
where $H$ does not depend on $t_0$. In fact, we can prove $H=0$
using the string equation and the KdV equation \eqref{firstkdv};
details can be found in Section 2 of \cite{LX7}. So equation
\eqref{kdv2} can be rewritten as
$$\langle\langle\tau_0\tau_1\rangle\rangle=\frac{1}{12}\langle\langle\tau_0^4\rangle\rangle
+\frac{1}{2}\langle\langle\tau_0^2\rangle\rangle\langle\langle\tau_0^2\rangle\rangle.$$
A minute's thought will convince you that this is just  the identity in Proposition
\ref{npt2}(ii) after we apply the string equation \eqref{str} and the dilaton equation \eqref{dil}.

The normalized $n$-point function $G$ has some nice vanishing
properties not possessed by the original $n$-point function $F$.

\begin{theorem} \cite{LX7} \label{nptcoeff} Let $\mathcal{C} \left(\prod_{j=1}^n
x_j^{d_j},p(x_1,\dots,x_n)\right)$ denote the coefficient of
$\prod_{j=1}^n x_j^{d_j}$ in a polynomial or formal power series
$p(x_1,\dots,x_n)$.
\begin{enumerate}
\item[i)] Let $k>2g-2+n$, $d_j\geq0$ and $\sum_{j=1}^n
d_j=3g-2+n-k$. Then
\begin{equation*}
\mathcal C\left(z^k\prod_{j=1}^{n}x_j^{d_j},
G_g(z,x_1,\dots,x_n)\right)=0.
\end{equation*}

\item[ii)] Let $d_j\geq0$, $\sum_{j=1}^n
d_j=g$ and $a=\#\{j\mid d_j=0\}$. Then
\begin{equation*}
\mathcal C\left(z^{2g-2+n}\prod_{j=1}^{n}x_j^{d_j},
G_g(z,x_1,\dots,x_n)\right)=\frac{1}{4^g\prod_{j=1}^n(2d_j+1)!!}.
\end{equation*}

\item[iii)] Let $d_j\geq0$, $\sum_{j=1}^n
d_j=g+1$, $a=\#\{j\mid d_j=0\}$ and $b=\#\{j\mid d_j=1\}$. Then
\begin{equation*}
\mathcal C\left(z^{2g-3+n}\prod_{j=1}^{n}x_j^{d_j},
G_g(z,x_1,\dots,x_n)\right)=\frac{2g^2+(2n-1)g+\frac{n^2-n}{2}-3+\frac{5a-a^2}{2}}{4^g\prod_{j=1}^n(2d_j+1)!!}.
\end{equation*}
\end{enumerate}
\end{theorem}

As an important application of the $n$-point functions, we proved
some new identities of descendent integrals, which led to a proof
of the Faber intersection number conjecture (see Section
\ref{tauto}). Some of these results (e.g. Theorem \ref{nptcoeff})
have also found applications in Zhou's important work \cite{Zh2} on
Hurwitz-Hodge integrals.

Next we give an
interesting combinatorial interpretation of $n$-point functions in
terms of summation over binary trees.

Recall that a binary tree $T$ is a tree such that each node $v\in
V(T)$ either has no children ($v\in L(T)$ is a leaf) or has two
children ($v\notin L(T)$).

Let $T$ be a binary tree. Let $n=|L(T)|$ be the number of leaves. We
assign an integer $g(v)\geq0$ to each node $v\in V(T)$ and label the
$n$ leaves with distinct values $\ell(v)\in\{1,\dots,n\}$. Then we
call such $T$ a ``weighted marked binary tree'' (abbreviated ``WMB
tree'') and call $g(T)=\sum_{v\in V(T)}g(v)$ the total weight of
$T$.

Here are all the WMB trees with $(g,n)=(2,2)$:

\begin{minipage}{0.7in}
\begin{displaymath}
  \vcenter{\tiny\xymatrix@C=1mm@R=0mm{
   &2& \\
    &  *={\bullet}\ar@{-}[ddl]\ar@{-}[ddr] &  \\
  & & \\
 *=<12pt>[o][F-]{1}  &   &
  *=<12pt>[o][F-]{2}  \\
   0&&0
         }}\,
\end{displaymath}
\end{minipage}
\begin{minipage}{0.7in}
\begin{displaymath}
  \vcenter{\tiny\xymatrix@C=1mm@R=0mm{
   &1& \\
    &  *={\bullet}\ar@{-}[ddl]\ar@{-}[ddr] &  \\
  & & \\
 *=<12pt>[o][F-]{1}  &   &
  *=<12pt>[o][F-]{2}  \\
   1&&0
         }}\,
\end{displaymath}
\end{minipage}
\begin{minipage}{0.7in}

\begin{displaymath}
  \vcenter{\tiny\xymatrix@C=1mm@R=0mm{
   &1& \\
    &  *={\bullet}\ar@{-}[ddl]\ar@{-}[ddr] &  \\
  & & \\
 *=<12pt>[o][F-]{1}  &   &
  *=<12pt>[o][F-]{2}  \\
   0&&1
         }}\,
\end{displaymath}
\end{minipage}
\begin{minipage}{0.7in}
\begin{displaymath}
  \vcenter{\tiny\xymatrix@C=1mm@R=0mm{
   &0& \\
    &  *={\bullet}\ar@{-}[ddl]\ar@{-}[ddr] &  \\
  & & \\
 *=<12pt>[o][F-]{1}  &   &
  *=<12pt>[o][F-]{2}  \\
   2&&0
         }}\,
\end{displaymath}
\end{minipage}
\begin{minipage}{0.7in}
\begin{displaymath}
  \vcenter{\tiny\xymatrix@C=1mm@R=0mm{
   &0& \\
    &  *={\bullet}\ar@{-}[ddl]\ar@{-}[ddr] &  \\
  & & \\
 *=<12pt>[o][F-]{1}  &   &
  *=<12pt>[o][F-]{2}  \\
   1&&1
         }}\,
\end{displaymath}
\end{minipage}
\begin{minipage}{0.7in}

\begin{displaymath}
  \vcenter{\tiny\xymatrix@C=1mm@R=0mm{
   &0& \\
    &  *={\bullet}\ar@{-}[ddl]\ar@{-}[ddr] &  \\
  & & \\
 *=<12pt>[o][F-]{1}  &   &
  *=<12pt>[o][F-]{2}  \\
   0&&2
         }}\,
\end{displaymath}
\end{minipage}
\smallskip

\begin{proposition} \cite{LX7}
Denote
by WMB$(g,n)$ the set of isomorphism classes of all WMB trees with
total weight $g$ and $n$ leaves. Then

\begin{multline*}
12^g\left(\prod_{j=1}^n x_j\right)\cdot(x_1+\dots+x_n)^2
F_g(x_1,\dots,x_n)\\= \sum_{T\in\ {\rm WMB}(g,n)} \prod_{v\in
V(T)}\frac{\left(|L(v)|-3+\sum\limits_{\substack{w\in D(v)\\w\neq
v}}2g(w)\right)!!}{\left(|L(v)|-1+\sum\limits_{w\in
D(v)}2g(w)\right)!!}\left(\sum_{w\in
L(v)}x_{\ell(w)}\right)^{3g(v)+1},
\end{multline*}
where $D(v)\subset V(T)$ is the set of all descendants of $v$ and
$L(v)=D(v)\cap L(T)$.
\end{proposition}

The idea of proof of the above formula is simple. We apply
Proposition \ref{npt2}(ii) repeatedly, until the right hand side
contains only one-point functions. Note that partitions of indices
are in one-to-one correspondence with binary trees.

\subsection{An effective recursion formulae of descendent integrals}

In \cite{LX7}  (Proposition 5.1), we proved a recursion formula
which explicitly expresses any descendent integral in terms of those
with strictly lower genus.
\begin{proposition}
Let $d_j\geq0$ and $\sum_{j=1}^n d_j=3g+n-3$. Then
\begin{multline*}
(2g+n-1)(2g+n-2)\langle\prod_{j=1}^n\tau_{d_j}\rangle_g\\
=\frac{2d_1+3}{12}\langle\tau_0^4\tau_{d_1+1}\prod_{j=2}^n\tau_{d_j}\rangle_{g-1}-\frac{2g+n-1}{6}\langle\tau_0^3\prod_{j=1}^n\tau_{d_j}\rangle_{g-1}\\
+\sum_{\{2,\dots,n\}=I\coprod
J}(2d_1+3)\langle\tau_{d_1+1}\tau_0^2\prod_{i\in
I}\tau_{d_i}\rangle_{g'}\langle\tau_0^2\prod_{i\in
J}\tau_{d_i}\rangle_{g-g'}\\
-\sum_{\{2,\dots,n\}=I\coprod
J}(2g+n-1)\langle\tau_{d_1}\tau_0\prod_{i\in
I}\tau_{d_i}\rangle_{g'}\langle\tau_0^2\prod_{i\in
J}\tau_{d_i}\rangle_{g-g'}.
\end{multline*}
\end{proposition}
By the string equation, we may assume all indices $d_j\geq1$; then
every non-zero descendent integral on the right hand side has genus
strictly less than $g$. So everything reduces to the following
well-known identity of genus zero intersection numbers:
$$\langle\tau_{d_1}\cdots\tau_{d_n}\rangle_0=\binom{n-3}{d_1,\dots,d_n}.$$

\vskip 30pt
\section{Hodge integrals}

In this section, we study the integrals of products of $\psi$ and
$\lambda$ classes on $\overline{\sM}_{g,n}$. Hodge integrals arise
naturally in the localization computation of Gromov-Witten theory
\cite{Ko2, GrPa}, since the Atiyah-Bott localization formula
expresses $\mathbb C^*$-equivariant classes over $\overline{\mathcal
M}_{g,n}(\mathbb P^r,d)$ in terms of summation over the fixed point
loci consisting of products of moduli spaces of marked curves, whose
summands are just the products of Hodge integrals.

\subsection{Faber's algorithm} Faber's algorithm \cite{Fa} reduces
the calculation of general Hodge integrals to those with pure $\psi$
classes.

\smallskip
\noindent {\bf Step 1: Eliminating $\kappa$ classes}

The fact that intersection numbers involving both $\kappa$ classes
and $\psi$ classes can be reduced to intersection numbers involving
only $\psi$ classes was already known to Witten \cite{Wi}, and has been
developed by Arbarello-Cornalba \cite{Ar-Co}, Faber \cite{Fa} and
Kaufmann-Manin-Zagier \cite{KMZ} into a beautiful combinatorial
formalism.

First we fix notation as in \cite{KMZ}. Consider the semigroup
$N^\infty$ of sequences ${\bold m}=(m_1,m_2,\dots)$ where $m_i$ are
nonnegative integers and $m_i=0$ for sufficiently large $i$.

Let $\bold m, \bold{a_1,\dots,a_n} \in N^\infty$, $\bold
m=\sum_{i=1}^n \bold{a_i}$.
$$|\bold m|:=\sum_{i\geq 1}i m_i,\quad ||\bold m||:=\sum_{i\geq1}m_i,\quad \binom{\bold m}{\bold{a_1,\dots,a_n}}:=\prod_{i\geq1}\binom{ m_i}{a_1(i),\dots,a_n(i)}.$$
We have the following formula \cite{KMZ} to remove $\kappa$ classes.

\begin{equation}\label{kmz}
\langle\prod_{j=1}^n\tau_{d_j}\kappa(\bold m)\Psi\rangle_g
=\sum_{k=0}^{||\bold m||}\frac{(-1)^{||\bold
m||-k}}{k!}\sum_{\substack {\bold
m=\bold{m_1}+\cdots+\bold{m_k}\\\bold{m_i}\neq\bold 0}}\binom{\bold
m}{\bold
{m_1,\dots,m_k}}\langle\prod_{j=1}^n\tau_{d_j}\prod_{j=1}^k\tau_{|\bold{m_j}|+1}\Psi\rangle_g,
\end{equation}
where $\kappa(\bold m)\triangleq \prod_{i\geq1}\kappa_i^{m(i)}$ and
$\Psi$ is the pull-back of classes from $\overline{\mathcal{M}}_{g}$
under forgetful morphisms, such as ${\rm ch}(\mathbb E)$ and
$\lambda$ classes.

\smallskip
\noindent {\bf Step 2: Substituting $\lambda$ classes by Chern
characters}

There is a universal formula to express $\lambda$ classes in terms
of ${\rm ch}(\mathbb E)$,
\begin{equation*}
\lambda_j = \sum_{\mu \vdash j}(-1)^{j-\ell(\mu)} \prod_{r \geq 1}
\frac{((r-1)!)^{m_r}}{m_r!}{\rm ch}_{\mu}({\mathbb E}), \quad j \geq
1,
\end{equation*}
where the sum ranges over all partitions $\mu$ of $j$, $\ell(\mu)$
is the length of $\mu$ and $m_r$ is the number of $r$ in $\mu$, and
${\rm ch}_{\mu}({\mathbb E})={\rm ch}_{\mu_1}(\mathbb E)\cdots {\rm
ch}_{\mu_\ell}(\mathbb E)$. Since ${\rm ch}_{2k}(\mathbb E)=0$ when
$k>0$, we may consider only partitions into odd numbers.

\smallskip
\noindent {\bf Step 3: Applying Mumford's formula}

So we arrive at the following integrals with only ${\rm ch}(\mathbb
E)$ and $\psi$ classes,
$$\int_{\overline{\mathcal{M}}_{g,n}}\psi_{1}^{d_{1}}\cdots\psi_{n}^{d_{n}}
{\rm ch}_{2k_1-1}(\mathbb E)\cdots{\rm ch}_{2k_\ell-1}(\mathbb E).$$
We will apply Mumford's formula \cite{Mu}
\begin{equation}\label{eqmu}
{\rm ch}_{2k-1}(\mathbb E)= \frac{ B_{2k}}{(2k)!}\left[
\kappa_{2k-1}-\sum_{i=1}^{n}\psi_i^{2k-1}+\frac12\sum_{\xi\in\Delta}{p_{\xi}}_*
\left(\sum_{i=0}^{2k-2}\psi_1^i(-\psi_2)^{2k-2-i}\right)\right],
\end{equation}
where $\Delta$ is the set of boundary divisors and $B_{2g}$ is the
$2g$-th Bernoulli number.

From \cite{Kn}, we know the pull-back behavior of Hodge bundles
under the two natural boundary gluing morphisms.
$$p: \overline{\mathcal M}_{g_1,n_1+1}\times\overline{\mathcal M}_{g_2,n_2+1}
\rightarrow\overline{\mathcal M}_{g_1+g_2,n_1+n_2},\quad p^*(\mathbb
E')= \mathbb E_1\oplus \mathbb E_2.$$
$$p: \overline{\mathcal M}_{g-1,n+2}\rightarrow\overline{\mathcal M}_{g,n},\quad p^*(\mathbb E')=\mathbb E\oplus\mathbb C.$$
Where $\mathbb E'$ is the Hodge bundle of the target moduli spaces and
$\mathbb C$ is the trivial line bundle.

Since ${\rm ch}(\mathbb E_1\oplus\mathbb E_2)={\rm ch}(\mathbb
E_1)+{\rm ch}(\mathbb E_2)$, we have
\begin{multline*}
\langle\prod_{i=1}^n\tau_{d_i}\prod_{i=1}^\ell{\rm
ch}_{2k_i-1}(\mathbb E)\rangle_g\\
=\frac{B_{2k_1}}{(2k_1)!!}\left(\langle\tau_{2k_1}\prod_{i=1}^n\tau_{d_i}\prod_{i=2}^\ell{\rm
ch}_{2k_i-1}(\mathbb
E)\rangle_g+\sum_{j=1}^n\langle\tau_{d_j+2k_1-1}\prod_{i\neq
j}\tau_{d_i}\prod_{i=2}^\ell{\rm ch}_{2k_i-1}(\mathbb
E)\rangle_g\right.\\
\left.+\frac12\sum_{j=0}^{2k_1-2}(-1)^j\langle\tau_j\tau_{2k_1-2-j}\prod_{i=1}^n\tau_{d_i}\prod_{i=2}^\ell{\rm
ch}_{2k_i-1}(\mathbb E)\rangle_{g-1}\right.\\
\left.+\frac12\sum_{\substack{I\coprod J=\underline{n}\\I'\coprod
J'=\{2,\dots,\ell\}}}\sum_{j=0}^{2k_1-2}(-1)^j\langle\tau_j\prod_{i\in
I}\tau_{d_i}\prod_{i\in I'}{\rm ch}_{2k_i-1}(\mathbb
E)\rangle_{g'}\langle\tau_{2k_1-2-j}\prod_{i\in
J}\tau_{d_i}\prod_{i\in J'}{\rm ch}_{2k_i-1}(\mathbb
E)\rangle_{g-g'} \right)
\end{multline*}

So we can reduce the integral to pure $\psi$ classes by induction on
the number of ${\rm ch}(\mathbb E)$. In fact, Faber's algorithm \cite{Fa} computes
more general intersection numbers, which may contain boundary divisors.

\subsection{Hodge integral formulae} There are very few
closed formulae for Hodge integrals. For example, Getzler and
Pandharipande \cite{GP} showed that the degree zero Virasoro
conjecture for $\mathbb P^1$, $\mathbb P^2$ and $\mathbb P^3$ implies
respectively the following three Hodge integral formulae:
\begin{align}
\label{lg}\langle\tau_{d_1}\cdots\tau_{d_n}\mid\lambda_g\rangle_g&=\binom{2g+n-3}{d_1,\dots,d_n}
\frac{2^{2g-1}-1}{2^{2g-1}}\frac{|B_{2g}|}{(2g)!},\\
\label{l2g}\langle\tau_{d_1}\cdots\tau_{d_n}\mid\lambda_g\lambda_{g-1}\rangle_g&=
\frac{(2g-3+n)!|B_{2g}|}{2^{2g-1}(2g)!\prod_{j=1}^{n}(2d_j-1)!!},\\
\label{l3g}\langle\lambda_{g-1}^3\rangle_g&=\frac{1}{(2g-2)!}\frac{|B_{2g-2}|}{2g-2}\frac{|B_{2g}|}{2g}.
\end{align}

We will see in Section \ref{tauto} that \eqref{l2g} is an equivalent
formulation of the Faber intersection number conjecture. The
$\lambda_g$ theorem \eqref{lg} was first proved by Faber and
Pandharipande \cite{FP2}. Goulden, Jackson and Vakil \cite{GJV2}
have a short proof using the ELSV formula. The identity \eqref{l3g} is
proved in \cite{FP1}.

Liu, Liu and Zhou \cite{LLZ2} gave a unified proof of
\eqref{lg},\eqref{l3g} as a consequence of the Mari\~{n}o-Vafa
formula.

The following Hodge integral identity is proved in our paper
\cite{LX6}.
\begin{proposition}
Let $g\geq2$, $d_j\geq1$ and $\sum_{j=1}^{n}(d_j-1)=g$. Then
\begin{multline*}
-\frac{(2g-2)!}{|B_{2g-2}|}\int_{\overline{\sM}_{g,n}}\psi_1^{d_1}\cdots\psi_n^{d_n}{\rm ch}_{2g-3}(\mathbb E)\\
=\frac{2g-2}{|B_{2g-2}|}\left(\int_{\overline{\sM}_{g,n}}\psi_1^{d_1}\cdots\psi_n^{d_n}\lambda_{g-1}\lambda_{g-2}-3\int_{\overline{\sM}_{g,n}}\psi_1^{d_1}\cdots\psi_n^{d_n}\lambda_{g-3}\lambda_{g}\right)\nonumber\\
=\frac{1}{2}\sum_{j=0}^{2g-4}(-1)^j\langle\tau_{2g-4-j}\tau_j\tau_{d_1}\cdots\tau_{d_n}\rangle_{g-1}+\frac{(2g-3+n)!}{2^{2g+1}(2g-3)!}\cdot\frac{1}{\prod_{j=1}^n(2d_j-1)!!}.
\end{multline*}
\end{proposition}

In fact, from Mumford's formula \eqref{eqmu}, the above Hodge
integral is equivalent to the following identity:
\begin{multline*}
\frac{(2g-3+n)!}{2^{2g+1}(2g-3)!\prod_{j=1}^n(2d_j-1)!!}
=\langle\tau_{2g-2}\prod_{j=1}^n\tau_{d_j}\rangle_g-\sum_{j=1}^{n}
\langle\tau_{d_j+2g-3}\prod_{i\neq j}\tau_{d_i}\rangle_g\\
+\frac{1}{2}\sum_{\underline{n}=I\coprod
J}\sum_{j=0}^{2g-4}(-1)^j\langle\tau_{j}\prod_{i\in
I}\tau_{d_i}\rangle_{g'}\langle\tau_{2g-4-j}\prod_{i\in
J}\tau_{d_i}\rangle_{g-g'},
\end{multline*}
which can be proved by packing the right-hand side as coefficients
of $n$-point functions; see discussions in Section \ref{tauto}.

The remarkable ELSV formula of Ekedahl, Lando, Shapiro, and
Vainshtein \cite{ELSV} relates single Hurwitz numbers to
intersection theory on the moduli space of curves.

\begin{theorem}(ELSV formula) Let $n=l(\mu)$ and $r=2g-2+|\mu|+ n$. Then
\begin{equation}\label{elsv} H_{g,\mu}= r! \prod_{i=1}^n \left(
\frac{ \mu_i^{\mu_i}} { \mu_i!} \right) \int_{\overline{\sM}_{g,n}}
\frac { 1 - \lambda_1 + \cdots + (-1)^g \lambda_g} { ( 1 - \mu_1
\psi_1) \cdots (1 - \mu_n \psi_n)},
\end{equation}
\end{theorem}

The ELSV formula was originally proved by studying the
degree of the Lyashko-Looijenga mapping, which can be expressed
in terms of the top Segre class of the completed Hurwitz space,
regarded as a cone over $\overline{\sM}_{g,n}$. The ELSV formula is more succinctly recovered
using virtual localization on moduli spaces of relative stable
morphisms \cite{GV, LLZ1}. It can also be derived as a limit of the
Mari\~no-Vafa formula \cite{LLZ2, Liu}.

A key step in Kazarian-Lando's proof \cite{KL} of the
Witten-Kontsevich theorem is to invert ELSV and eliminate $\lambda$
classes. So assertions about descendent integrals may be proved by
studying Hurwitz numbers.

\begin{proposition}(Kazarian-Lando) \cite{KL}\label{kl}
Let $\sum_{i=1}^n d_i=3g-3+n$. Then
$$
\langle\tau_{d_1}\cdots\tau_{d_n}\rangle_g=
\sum_{\mu_1=1}^{d_1+1}\dots \sum_{\mu_n=1}^{d_n+1}\left(\frac{|{\rm
Aut}(\mu)|}{(2g-2+|\mu|+n)!}
\prod_{i=1}^n\frac{(-1)^{d_i+1-\mu_i}}{(d_i+1-\mu_i)!\mu_i^{\mu_i-1}}\right)
H_{g,\mu}.
$$
\end{proposition}

In \cite {GJV1}, Goulden, Jackson and Vakil proposed a conjectural ELSV-type
formula expressing one-part double Hurwitz numbers in terms of
intersection theory on some compactified universal Picard variety. Recently, D. Zvonkine
has proposed a conjectural ELSV-type formula on moduli spaces of $r$-spin curves.

\vskip 30pt
\section{Higher Weil-Petersson volumes}
For $\bold b\in N^\infty$, we denote by $V_{g,n}(\bold b)$ the
higher Weil-Petersson volume $$V_{g,n}(\bold b):= \langle\tau_0^n
\kappa(\bold b)\rangle_g=\int_{\overline{\sM}_{g,n}}\kappa(\bold
b).$$ Also we write $V_g(\bold b)$ instead of $V_{g,0}(\bold b)$. If
$\bold b=(3g-3+n,0,0,\dots)$, we get the classical Weil-Petersson
volumes.

Higher Weil-Petersson volumes were extensively studied in the paper
\cite{KMZ}. In particular, they obtained a closed formula for
$V_{0,n}(\bold b)$.

\subsection{Generalization of Mirzakhani's recursion formula}
In a series of innovative papers \cite{Mir,Mir2}, Mirzakhani
utilizes hyperbolic geometry to obtain a beautiful recursion formula
of the Weil-Petersson volumes of the moduli spaces of bordered
Riemann surfaces. By taking derivatives in Mirzakhani's recursion,
Mulase and Safnuk \cite{MS} obtained a differential form of
Mirzakhani's recursion formula involving integrals of $\kappa_1$ and
$\psi$ classes on moduli spaces of curves, which is immediately seen
to imply the DVV formula \eqref{dvv}.

Wolpert's formula \cite{Wo} tells us that
$$\kappa_1=\frac{1}{2\pi^2}\omega_{WP},$$
where $\omega_{WP}$ is the Weil-Petersson K\"ahler form. However,
Wolpert's formula has no counterpart for higher degree $\kappa$
classes.

In the papers \cite{LX2, LX3}, we have proved that the Mulase-Safnuk form of Mirzakhani's
recursion formula is in fact equivalent to the Witten-Kontsevich
theorem. The proof can be generalized to prove an analogue of  the Mulase-Safnuk form of Mirzakhani's
recursion containing arbitrary higher degree $\kappa$ classes.

\begin{theorem}\cite{LX2} \label{wp1} Let
$\bold b\in N^\infty$ and $d_j\geq 0$. Then
\begin{multline}\label{eqwp}
(2d_1+1)!!\langle\kappa(\bold b)\tau_{d_1}\cdots\tau_{d_n}\rangle_g\\
=\sum_{j=2}^n\sum_{\bold L+\bold{L'}=\bold b}\alpha_{\bold
L}\binom{\bold b}{\bold L}\frac{(2(|\bold
L|+d_1+d_j)-1)!!}{(2d_j-1)!!}\langle\kappa(\bold{L'})\tau_{|\bold
L|+d_1+d_j-1}\prod_{i\neq
1,j}\tau_{d_i}\rangle_g\\
+\frac{1}{2}\sum_{\bold L+\bold{L'}=\bold b}\sum_{r+s=|\bold
L|+d_1-2}\alpha_{\bold
L}\binom{\bold b}{\bold L}(2r+1)!!(2s+1)!!\langle\kappa(\bold{L'})\tau_r\tau_s\prod_{i=2}^n\tau_{d_i}\rangle_{g-1}\\
+\frac{1}{2}\sum_{\substack{\bold L+\bold{e}+\bold{f}=\bold
b\\I\coprod J=\{2,\dots,n\}}}\sum_{r+s=|\bold L|+d_1-2}\alpha_{\bold
L}\binom{\bold b}{\bold
L,\bold{e},\bold{f}}(2r+1)!!(2s+1)!!\\
\times \langle\kappa(\bold{e})\tau_r\prod_{i\in
I}\tau_{d_i}\rangle_{g'}\langle\kappa(\bold{f})\tau_s\prod_{i\in
J}\tau_{d_i}\rangle_{g-g'},
\end{multline}
where the constants $\alpha_{\bold L}$ are determined recursively
from the following formula
$$\sum_{\bold L+\bold{L'}=\bold b}\frac{(-1)^{||\bold L||}\alpha_{\bold L}}{\bold L!\bold{L'}!(2|\bold{L'}|+1)!!}=0,\qquad \bold b\neq0,$$
namely
$$\alpha_{\bold b}=\bold b!\sum_{\substack{\bold L+\bold{L'}=\bold b\\ \bold{L'}\neq\bold 0}}\frac{(-1)^{||\bold L'||-1}\alpha_{\bold L}}{\bold L!\bold{L'}!(2|\bold{L'}|+1)!!},\qquad\bold b\neq0,$$
with the initial value $\alpha_{\bold 0}=1$.
\end{theorem}

Denoting $\alpha(\ell,0,0,\dots)$ by $\alpha_\ell$, we recover
Mirzakhani's recursion formula with
$$\alpha_\ell=l!\beta_\ell=(-1)^{\ell-1}(2^{2\ell}-2)\frac{B_{2\ell}}{(2\ell-1)!!}.$$
We also have
$$\alpha(\bm{\delta}_\ell)=\frac{1}{(2\ell+1)!!},$$
where $\bm{\delta}_\ell$ denotes the sequence with $1$ at the
$\ell$-th place and zeros elsewhere.

Note that Theorems \ref{wp1} holds only for $n\geq 1$. If $n=0$,
i.e. for higher Weil-Petersson volumes of $\overline{\mathcal M}_g$,
we may apply the following formula first, which is a special case of
Proposition 3.1 of the paper \cite{LX2}.

\begin{equation}
\langle\kappa(\bold b)\rangle_g=\frac{1}{2g-2}\sum_{\bold L+\bold
L'=\bold b}(-1)^{||\bold L||}\binom{\bold b}{\bold
L}\langle\tau_{|\bold L|+1}\kappa(\bold L')\rangle_g.
\end{equation}

So we can use Theorems \ref{wp1} to compute any intersection numbers
of $\psi$ and $\kappa$ classes recursively with the three initial
values
$$\langle\tau_0\kappa_1\rangle_1=\frac{1}{24},\qquad
\langle\tau_0^3\rangle_0=1,\qquad
\langle\tau_1\rangle_1=\frac{1}{24}.$$

The idea to look at the reciprocal of $\alpha_{\bold L}$ was
inspired by the work of Mulase and Safnuk \cite{MS}, where they
considered the reciprocal of $\alpha_{\ell}$.

For the proof of Theorem \ref{wp1}, we first transfer the inversion of constants
$\alpha_{\bold L}$ to the left hand side of \eqref{eqwp}, then we
carry out the computation by applying the formula \eqref{kmz} and
conclude the result from the DVV formula \eqref{dvv}.

\subsection{Recursion formulae of higher Weil-Petersson volumes}

Although there are recursion formulae for higher Weil-Petersson
volumes in genus zero \cite{KMZ, Zo}, it seems difficult to
generalize the methods of these papers to deduce explicit recursion
formulae between $V_{g,n}(\bold b)$ valid in all genera. We remark
that Zograf's elegant algorithm \cite{Zo2} for computing
Weil-Petersson volumes seems also not easy to generalize to higher
degree $\kappa$ classes.

The following recursion formulae are proved in
\cite{LX2}.

\begin{proposition}\cite{LX2} \label{wp2}
Let $\bold b\in N^\infty$ and $n\geq 1$. Then
\begin{multline*}
\big(2g-1+||\bold b||\big)V_{g,n}(\bold
b)=\frac{1}{12}V_{g-1,n+3}(\bold b)-\sum_{\substack{\bold L+\bold{L'}=\bold b\\
||\bold{L'}||\geq 2}}\binom{\bold b}{\bold L}V_{g,n}(\bold
L+\bm\delta_{|\bold{L'}|})\\
+\frac{1}{2}\sum_{\substack{\bold L+\bold{L'}=\bold b\\
\bold{L}\neq\bold 0, \bold{L'}\neq\bold
0}}\sum_{r+s=n-1}\binom{\bold b}{\bold
L}\binom{n-1}{r}V_{g',r+2}(\bold L)V_{g-g',s+2}(\bold {L'}).
\end{multline*}
\end{proposition}

Proposition \ref{wp2} is an effective formula for computing higher
Weil-Petersson volumes recursively by induction on $g$ and $||\bold
b||$, with initial values
$$V_{0,3}(0)=1 \quad \text{ and }\quad V_{0,n}(\bm{\delta}_{n-3})=1,\ n\geq 4.$$

\begin{proposition}\cite{LX2}\label{wp3}
Let $g\geq2$ and $\bold b\in N^\infty$. Then
\begin{multline*}
\big((2g-1)(2g-2)+(4g-3)||\bold b||+||\bold b||^2\big)V_{g}(\bold
b)=5\sum_{\bold L+\bold{L'}=\bold b}\binom{\bold b}{\bold
L}V_{g,1}(\bold
L+\bm\delta_{|\bold{L'}|+1})\\
-\frac{1}6{}\sum_{\bold L+\bold{L'}=\bold b}\binom{\bold b}{\bold
L}V_{g-1,3}(\bold L+\bm\delta_{|\bold{L'}|})-\sum_{\bold
L+\bold{e}+\bold{f}=\bold b}\binom{\bold b}{\bold{
L,e,f}}V_{g',1}(\bold e+\bm\delta_{|\bold
L|})V_{g-g',2}(\kappa(\bold
f))\\
-(2g-1+||\bold b||)\sum_{\substack{\bold L+\bold{L'}=\bold
b\\||\bold{L'}||\geq 2}}\binom{\bold b}{\bold L}V_{g}(\bold
L+\bm\delta_{|\bold{L'}|})\\
-\sum_{\substack{\bold L+\bold{L'}=\bold b\\||\bold{L'}||\geq
2}}\binom{\bold b}{\bold L}\sum_{\bold e+\bold f=\bold
L+\bm{\delta}_{|\bold{L'}|}}\binom{\bold
L+\bm{\delta}_{|\bold{L'}|}}{\bold e} V_{g}(\bold
e+\bm\delta_{|\bold f|}).
\end{multline*}
\end{proposition}

By induction on $||\bold b||$, Proposition \ref{wp3} reduces the
computation of $V_{g}(\bold b)$ to the cases of $V_{g,n}(\bold b)$
for $n\geq1$, which have been computed by Proposition \ref{wp2}.
Therefore Propositions \ref{wp2} and \ref{wp3} completely determine
higher Weil-Petersson volumes of moduli spaces of curves.

The virtue of the above recursion formulae is that they do not
involve $\psi$ classes. So if one is only interested in computing
higher Weil-Petersson volumes, the above recursions are more
efficient both in speed and space, especially when we utilize
``option remember'' in Maple procedures.

On the other hand, we know that intersection numbers of mixed $\psi$
and $\kappa$ classes can be expressed by integrals of pure $\kappa$
classes using the following well-known formula repeatedly. We give a
proof here, since the same argument is also used in the proof of Propositions \ref{wp2} and
\ref{wp3}.
\begin{proposition} Let $d_n\geq1$. Then
\begin{equation*}
\langle\tau_{d_1}\cdots\tau_{d_n}\kappa(\bold b)\rangle_g\\
=\sum_{\bold L+\bold{L'}=\bold b}\binom{\bold b}{\bold
L}\langle\tau_{d_1}\cdots\tau_{d_{n-1}}\kappa(\bold{L'})\kappa_{|\bold
L|+d_n-1}\rangle_g.
\end{equation*}
\end{proposition}
\begin{proof}
We need some results from \cite{Ar-Co}. Let $\pi_{n}: \overline{\sM}_{g,n}\longrightarrow
\overline{\sM}_{g,n-1}$ be the morphism that forgets the last marked
point. Then
\begin{enumerate}
\item[i)] $\pi_{n*}(\psi_1^{d_1}\cdots\psi_{n-1}^{d_{n-1}}\psi_n^{d_{n}})=
\psi_1^{d_1}\cdots\psi_{n-1}^{d_{n-1}}\kappa_{d_n-1}$\qquad
\text{for}\quad $d_n\geq1$,

\item[ii)] $\kappa_a=\pi^*_{n}(\kappa_a)+\psi^a_{n}$\qquad
\text{on}\quad $\overline{\sM}_{g,n}$.
\end{enumerate}

So if $d_n>0$, by the projection formula we have
\begin{align*}
\langle\tau_{d_1}\cdots\tau_{d_n}\kappa(\bold b)\rangle_g&=
\int_{\overline{\sM}_{g,n-1}}\pi_{n*}\left(\psi_1^{d_1}\cdots\psi_{n}^{d_n}\prod_{i\geq
1}(\pi_{n}^*\kappa_i+\psi_{n}^i)^{b(i)}\right) \\&=\sum_{\bold
L+\bold{L'}=\bold b}\binom{\bold b}{\bold
L}\langle\tau_{d_1}\cdots\tau_{d_{n-1}}\kappa(\bold{L'})\kappa_{|\bold{L}|+d_n-1}\rangle_g
\end{align*}

\end{proof}

\vskip 30pt
\section{Faber's conjecture on tautological rings}\label{tauto}

Denote by $\sM_g$ the moduli space of Riemann surfaces of genus
$g\geq2$. The tautological ring $\mathcal R^*(\sM_g)$ is defined to
be the $\mathbb Q$-subalgebra of the Chow ring $\mathcal A^*(\sM_g)$
generated by the tautological classes $\kappa_i$ and $\lambda_i$.

\begin{proposition} $\mathcal R^*(\sM_g)$ has the following
properties:
\begin{enumerate}
\item[i)] {\rm (Mumford \cite{Mu})} $\mathcal R^*(\sM_g)$ is in fact generated by
the $g-2$ classes $\kappa_1,\dots,\kappa_{g-2}$;

\item[ii)] {\rm (Looijenga \cite{Lo})} $\mathcal R^j(\sM_g)=0$ for $j>g-2$
and $\dim \mathcal R^{g-2}(\sM_g)\leq1$ {\rm (}Faber \cite{Fa2} showed
that actually $\dim\mathcal R^{g-2}(\sM_g)=1${\rm)}.
\end{enumerate}
\end{proposition}

Around 1993, Faber \cite{Fa} proposed a series of remarkable
conjectures about the structure of the tautological ring $\mathcal
R^*(\sM_g)$. Faber's conjectures have aroused a lot of interest and
motivated tremendous progress toward understanding the topology of
the moduli space of curves.

Faber's conjecture is mentioned as a fundamental question in
monographs such as \cite{HM} (pp. 68-70) and \cite{Ga} (pp.
148-155).

Roughly speaking, Faber's conjecture asserts that ``$\mathcal R^*(\sM_g)$
behaves like the cohomology ring of a $(g-2)$-dimensional complex
projective manifold.'' We now state it precisely.

\begin{enumerate}
\item[i)] (Perfect pairing conjecture) When an isomorphism $\mathcal R^{g-2}(\sM_g)\cong\mathbb Q$ is
fixed, the following natural pairing is perfect
\begin{equation}\label{perf}
\mathcal R^{k}(\mathcal M_g)\times \mathcal R^{g-2-k}(\mathcal
M_g)\longrightarrow \mathcal R^{g-2}(\mathcal M_g)=\mathbb Q;
\end{equation}
Faber's perfect pairing conjecture is still open to this day.

\item[ii)] The $[g/3]$ classes $\kappa_1,\dots,\kappa_{[g/3]}$
generate the ring, with no relations in degrees $\leq [g/3]$;
\end{enumerate}

The part (ii) of Faber's conjecture has been proved by Morita
\cite{Mo} and Ionel \cite{Io}.

Another important part of Faber's conjecture is the intersection
number conjecture, which we will discuss in some detail.

\subsection{The Faber intersection number conjecture}  Faber predicted the top intersections as
 the following relations in
$\mathcal R^{g-2}(\sM_g)$,
\begin{equation}\label{fa1}
\pi_*(\psi_1^{d_1+1}\cdots\psi_n^{d_n+1})=\frac{(2g-3+n)!(2g-1)!!}{(2g-1)!\prod_{j=1}^{n}(2d_j+1)!!}\kappa_{g-2},\quad\text{for
}\sum_{j=1}^n d_j=g-2,
\end{equation}
where $ \pi: \overline{\sM}_{g,n}\longrightarrow \overline{\sM}_{g}$
is the forgetful morphism.

Thus the Faber intersection number conjecture determines the ring
structure of $\mathcal R^*(\sM_g)$ if Faber's perfect pairing
conjecture is true.

Since $\lambda_g\lambda_{g-1}$ vanishes on the boundary of
$\overline{\sM}_g$, the Faber intersection number conjecture is
equivalent to
\begin{equation}\label{fa2}
\int_{\overline{\sM}_{g,n}}\psi_1^{d_1}\cdots\psi_n^{d_n}\lambda_g\lambda_{g-1}=
\frac{(2g-3+n)!|B_{2g}|}{2^{2g-1}(2g)!\prod_{j=1}^{n}(2d_j-1)!!}.
\end{equation}
By Mumford's formula \cite{Mu} for the Chern character of the Hodge
bundle, the above identity is equivalent to
\begin{align}\label{fa3}
\frac{(2g-3+n)!}{2^{2g-1}(2g-1)!\prod_{j=1}^n(2d_j-1)!!}
=&\langle\tau_{2g}\prod_{j=1}^n\tau_{d_j}\rangle_g-\sum_{j=1}^{n}
\langle\tau_{d_j+2g-1}\prod_{i\neq j}\tau_{d_i}\rangle_g\nonumber\\
&+\frac{1}{2}\sum_{j=0}^{2g-2}(-1)^j\langle\tau_{2g-2-j}\tau_j\prod_{i=1}^n\tau_{d_i}\rangle_{g-1}\\\nonumber
&+\frac{1}{2}\sum_{\underline{n}=I\coprod
J}\sum_{j=0}^{2g-2}(-1)^j\langle\tau_{j}\prod_{i\in
I}\tau_{d_i}\rangle_{g'}\langle\tau_{2g-2-j}\prod_{i\in
J}\tau_{d_i}\rangle_{g-g'},
\end{align}
where $d_j\geq1$, $\sum_{j=1}^{n}d_j=g+n-2$.

The identity \eqref{fa2} was shown to follow from the degree 0
Virasoro conjecture for $\mathbb P^2$ by Getzler and Pandharipande
\cite{GP}. Givental \cite{Giv} has announced a proof of the Virasoro
conjecture for $\mathbb P^n$. Y.-P. Lee and R. Pandharipande are
writing a book \cite{LP} giving details. Recently Teleman \cite{Te}
announced a proof of the Virasoro conjecture for all manifolds with semi-simple
quantum cohomology. His argument depends crucially on the Mumford
conjecture about the stable rational cohomology rings of the moduli
spaces proved by Madsen and Weiss.

However, the Virasoro conjecture is a huge machinery and conceals the
combinatorial structure of intersection numbers. The proof of the
Mumford conjecture is also highly nontrivial. So a more direct proof
of the Faber intersection number conjecture is very much desired.

Goulden, Jackson and Vakil \cite{GJV3} recently give an enlightening
proof of the identity \eqref{fa1} for up to three points. Their
remarkable proof relied on relative virtual localization in
Gromov-Witten theory and some tour de force combinatorial
computations.

\subsection{Relations with $n$-point functions} Now we describe our approach to proving the identity \eqref{fa3}; the
details are in \cite{LX6}.

Since the one- and two-point functions in genus $0$ are
$$F_0(x)=\frac{1}{x^2},\qquad F_{0}(x,y)=\frac{1}{x+y}=\sum_{k=0}^\infty (-1)^k\frac{x^k}{y^{k+1}},$$
it is consistent to define the virtual intersection numbers
$$\langle\tau_{-2}\rangle_0=1,\qquad \langle\tau_{k}\tau_{-1-k}\rangle_0=(-1)^k,\ k\geq0.$$

For $a, b\in\mathbb Z$, we introduce the following notation:
$$
L_g^{a,b}(y,x_1\dots,x_n) \triangleq
\sum_{g'=0}^g\sum_{\underline{n}=I\coprod J} (y+\sum_{i\in I}x_i)^a
(-y+\sum_{i\in J}x_i)^b F_{g'}(y,x_I)F_{g-g'}(-y,x_J),
$$
We regard $L_g^{a,b}(y,x_1\dots,x_n)$ as a formal series in $\mathbb
Q[x_1,\dots,x_n][[y,y^{-1}]]$ with $\deg y<\infty$.

\begin{multline*}
\frac{1}{2}\sum_{\underline{n}=I\coprod
J}\sum_{j=0}^{2g-2}(-1)^j\langle\tau_{j}\prod_{i\in
I}\tau_{d_i}\rangle_{g'}\langle\tau_{2g-2-j}\prod_{i\in
J}\tau_{d_i}\rangle_{g-g'}
+\langle\prod_{j=1}^n\tau_{d_j}\tau_{2g}\rangle_g- \sum_{j=1}^{n}
\langle\tau_{d_j+2g-1}\prod_{i\neq j}\tau_{d_i}\rangle_g\\
=\frac{1}{2}\sum_{\underline{n}=I\coprod J}\sum_{j\in\mathbb
Z}(-1)^j\langle\tau_{j}\prod_{i\in
I}\tau_{d_i}\rangle_{g'}\langle\tau_{2g-2-j}\prod_{i\in
J}\tau_{d_i}\rangle_{g-g'}\\
=\left[\sum_{g'=0}^g\sum_{\underline{n}=I\coprod J}
F_{g'}(y,x_I)F_{g-g'}(-y,x_J)\right]_{y^{2g-2}\prod_{i=1}^n
x_i^{d_i}}\\
=\left[L_g^{0,0}(y,x_1,\dots,x_n)\right]_{y^{2g-2}\prod_{i=1}^n
x_i^{d_i}}.
\end{multline*}

The right-hand side of \eqref{fa3} may be written as the
coefficients of the $n$-point functions.
\begin{equation}
\left[F_{g-1}(y,-y,x_1,\dots,x_n)\right]_{y^{2g-2}}+\left[L_g^{0,0}(y,x_1,\dots,x_n)\right]_{y^{2g-2}\prod_{i=1}^n
x_i^{d_i}}.
\end{equation}
So in order to prove the Faber intersection number conjecture, it is
sufficient to prove the following results.

\begin{proposition} We have
\begin{enumerate}
\item[i)]
$$\left[L^{0,0}_g(y,x_1\dots,x_n)\right]_{y^{2g-2}}=0;$$

\item[ii)] For
$d_j\geq 1$ and $\sum_{j=1}^n d_j=g+n$,
$$\left[L^{2,2}_g(y,x_1\dots,x_n)\right]_{y^{2g}\prod_{j=1}^n
x_j^{d_j}}= \frac{(2g+n+1)!}{4^g(2g+1)!\prod_{j=1}^n(2d_j-1)!!}.$$
\end{enumerate}
\end{proposition}

Next we apply Proposition \ref{npt2}(ii) to prove a strengthened
version of the above proposition inductively.

\section{Dimension of tautological rings}

If $\bold m\in N^\infty$ and $|\bold m|=g-2$, then, from Faber's
intersection number identity \eqref{fa1} and equation \eqref{kmz}, we have in $\mathcal
R^{g-2}(\mathcal M_g)$
\begin{equation*}
\kappa(\bold m)=\sum_{r=1}^{||\bold m||}\frac{(-1)^{||\bold
m||-r}}{r!}\sum_{\substack {\bold
m=\bold{m_1}+\cdots+\bold{m_r}\\\bold{m_i}\neq\bold 0}}\binom{\bold
m}{\bold
{m_1,\dots,m_r}}\frac{(2g-3+r)!\kappa_{g-2}}{(2g-2)!!\prod_{j=1}^r(2|\bold
m_j|+1)!!}.
\end{equation*}
Let $0\leq k\leq g-2$ and denote by $p(k)$ the partition number of $k$. Define a matrix $V_g^k$ of size $p(k)\times
p(g-2-k)$ with entries $$(V_g^k)_{\bold L, \bold
L'}=\sum_{r=1}^{||\bold L+\bold L'||}\frac{(-1)^{||\bold L+\bold
L'||-r}}{r!}\sum_{\substack {\bold L+\bold
L'=\bold{m_1}+\cdots+\bold{m_r}\\\bold{m_i}\neq\bold 0}}\binom{\bold
L+\bold L'}{\bold
{m_1,\dots,m_r}}\frac{(2g-3+r)!}{\prod_{j=1}^r(2|\bold m_j|+1)!!},$$
where $\bold L, \bold L'\in N^\infty$ and $|\bold L|=k$, $|\bold
L'|=g-2-k$.

We call $V_g^k$ the Faber intersection matrix. Instead of using the above closed formula directly, we \cite{LX8} have
some recursive ways to compute entries of $V_g^k$. As a result, we have computed $V_g^k$ for all $g\leq36$. We are interested in
the rank of the Faber intersection matrix, so we introduce the notation
$$R_g^k:=\rank V_g^k,\qquad R_g:=\sum_{k=0}^{g-2}R_g^k.$$
Obviously we have $R_g^k=R_g^{g-2-k},\ 0\leq k\leq g-2$. If Faber's
perfect pairing conjecture \eqref{perf} is true, then we have for
$0\leq k\leq g-2$,
$$R_g^k=\dim(R^k(\sM_g)),$$
$$R_g=\dim\mathcal R^*(\mathcal M_g).$$

Since Faber has verified his conjecture for all $g\leq23$, the above
relations hold at least when $g\leq23$. Like the importance of
cohomology, the dimensions of tautological rings are important
invariants of moduli spaces of curves.

\subsection{Ramanujan's mock theta functions}
First we recall the standard notation of basic hypergeometric
series.

$$(a)_n=(a;q)_n=(1-a)(1-aq)\cdots(1-aq^{n-1}),\qquad n\geq1$$
$$(a)_{0}=1,\qquad (a)_{\infty}=(a;q)_{\infty}=\prod_{n=0}^{\infty}(1-aq^n).$$

Recall Ramanujan's third order mock theta function $\omega(q)$

\begin{multline}\label{eqmock}
\omega(q)=\sum_{n=0}^\infty
\omega(n)q^n:=\sum_{n=0}^\infty\frac{q^{2n^2+2n}}{\prod_{j=0}^n(1-q^{2j+1})^2}\\
=1+2q+3q^2+4q^3+6q^4+8q^5+10q^6+14q^7+18q^8+22q^9\\
+29q^{10}+36q^{11}+44q^{12}+56q^{13}+68q^{14}+82q^{15}+\cdots
\end{multline}

In 1966, Andrews proved asymptotic formulae for $\omega(n)$ and
another third order mock theta function $f(q)$,
\begin{multline}
f(q)=\sum_{n=0}^\infty f(n)q^n:=\sum_{n=0}^\infty\frac{q^{n^2}}{\prod_{j=1}^n(1+q^{j})^2}\\
=1+q-2q^2+3q^3-3q^4+3q^5-5q^6+7q^7-6q^8+6q^9+\cdots
\end{multline}
 In 2003, Andrews \cite{An}
improved his asymptotic formulae for $f(n)$ to a conjectural exact
formula. Around the same time, Zwegers \cite{Zw} found a relationship
between mock theta functions and vector-valued modular forms.
Andrews' conjectural exact formula for $f(n)$ was proved by Bringmann
and Ono \cite{BO} using Zwegers' results, along with the theory of
Maass forms and Poincar\'{e} series. Recently, Garthwaite \cite{Gar}
proved an analogue of Andrews' exact formula for $\omega(n)$
following the method of Bringmann and Ono.

We first introduce some notation. Let $e(x):=e^{2\pi i x}$. If
$k\geq 1$  and $n$ are integers, define
\begin{equation}\label{eqakn}
A_k(n):=\frac{1}{2} \sqrt{\frac{k}{12}} \sum_{\substack{x \pmod {24k}\\
x^2 \equiv -24n+1 \pmod{24k}}}  \chi_{12}(x) \cdot
e\left(\frac{x}{12k}\right),
\end{equation}
where the sum runs over the residue classes modulo $24k$, and where
we use the Kronecker symbol
$$\chi_{12}(x) := \left(\frac{12}{x}\right)=\begin{cases}0, & x\text{ is not coprime to } 6,\\
1, & x\equiv\pm1\mod 12,\\
-1, & x\equiv\pm5\mod 12.
\end{cases}
$$
Let $I_{1/2}$ be the $I$-Bessel function
$$I_{1/2}(z)=\left(\frac{2}{\pi z}\right)^{\frac12}\sinh z.$$

Now we can state Garthwaite's formula \cite{Gar}:
\begin{equation}\label{eqgar}
\omega(n)=\frac{\pi}{2\sqrt{2}}(3n+2)^{-1/4}\sum_{k=1}^{\infty}
\frac{ (-1)^{k-1}A_{2k-1}\left(nk-\frac{3k(k-1)}{2}\right)}{2k-1}
\cdot I_{1/2}\left(\frac{\pi \sqrt{3n+2}}{6k-3}\right).
\end{equation}
Actually this series approaches the exact value very rapidly and can
be effectively used to compute exact values of $\omega(n)$.

\begin{lemma}\label{even}
When $n$ is odd, $\omega(n)$ is even.
\end{lemma}
\begin{proof} First we note that $$\frac{1}{(1-q^{2j+1})^2}=\sum_{k=0}^\infty
(k+1)q^{(2j+1)k}.$$ The lemma follows easily from \eqref{eqmock}.
\end{proof}

Compared with its definition \eqref{eqmock}, $\omega(q)$ has a
simpler expression. The following lemma is well-known to experts, but we
include a proof here for the reader's convenience.
\begin{lemma}
$$\sum_{n=0}^\infty
\omega(n)q^n=\sum_{n=0}^\infty\frac{q^n}{\prod_{j=0}^n(1-q^{2j+1})}.$$
\end{lemma}
\begin{proof}

Consider the following $q$-series,
$$F(t)=\sum_{n=0}^{\infty}\frac{t^n}{(tq)_n} \ ,$$
where $|q|<1$. Since $(tq)_n=(1-tq^n)(tq)_{n-1}$, we have
\begin{align*}
(1-t)F(t)&=1+\sum^{\infty}_{n=1}(\frac{1}{(tq)_n}-\frac{1}{(tq)_{n-1}})t^n\\
&=1+\sum^{\infty}_{n=1}(\frac{1}{(tq)_n}-\frac{(1-tq^n)}{(tq)_{n}})t^n\\
&=1+t\sum^{\infty}_{n=1}\frac{(tq)^n}{(tq)_n}\\
&=1+\frac{t^{2}q}{1-tq}F(tq).
\end{align*}
In the last equation, we used $(tq)_n=(1-tq)(tq^2)_{n-1}$.

Let $R(t)=(1-t)F(t)$. Then$$R(t)=1+\frac{t^2q}{(1-tq)(1-tq)}R(tq)$$

By iteration, we
have$$R(t)=1+\sum_{n=0}^{r-1}\frac{t^{2n+2}q^{(n+1)^2}}{(tq)^2_{n+1}}+\prod_{n=0}^{r}\frac{t^2q^{2n+1}}{(1-tq^{n+1})^2}R(tq^{r+1})$$

Letting $r\rightarrow\infty$, we
get $$(1-t)F(t)=\sum_{n=0}^\infty\frac{t^{2n}q^{n^2}}{(tq)^2_n}$$

Substituting $q$ by $q^2$, $t$ by $q$, we get the desired equation,
$$\sum_{n=0}^\infty\frac{q^n}{(1-q)(1-q^3)\cdots(1-q^{2n+1})}=\sum_{n=0}^\infty\frac{q^{2n^2+2n}}{(1-q)^2(1-q^3)^2\cdots(1-q^{2n+1})^2}$$

\end{proof}

\subsection{Asymptotics of tautological dimensions}

With the help of the website ``The On-Line Encyclopedia of Integer
Sequences'', we discovered the surprising coincidence that
$R_g=\omega(g-2)$ for $2\leq g\leq 17$. However, when $g>17$, this is
no longer true. Let us use the notation $\omega_g:=\omega(g-2)$.
$$
\begin{array}{c||c|c|c|c|c|c|c|c|c|c|c|c|c}
    g                               & 18 & 19&20 &21 &22 &23 &24 &25 &26 &27 &28 &29 &30
\\\hline\hline \omega_g & 101&122&146&176&210&248&296&350&410&484&566&660&772
\\\hline      R_g                   &
102&122&146&178&211&250&300&352&415&492&574&670&788
\end{array}
$$

In general, we have
the following conjecture.
\begin{conjecture}\label{dimen} For all $g\geq 2$, we have
\begin{equation}\label{dimineq}
R_g\geq\omega_g
\end{equation}
 and there
exists some constant $C>0$, such that
$$\lim_{g\rightarrow\infty}\frac{\omega_g}{R_g}=C.$$
\end{conjecture}


Faber's computation reveals that there should exist a uniquely
determined integer sequence $a(n)$ with $a(n)=0$ for $n\leq0$, such
that
$$
\dim\mathcal R^k(\mathcal M_g)=\begin{cases} p(k)-a(3k-g), & 0\leq k\leq \frac{g-2}{2},\\
\dim\mathcal R^{g-2-k}(\mathcal M_g), & \frac{g-2}{2}< k\leq g-2.
\end{cases}
$$

Thanks to Faber's verification, we know that $R_g^k=\dim\mathcal
R^{g-2-k}(\mathcal M_g)$ for at least $g\leq23$. Faber computed
$a(n)$ for $n\leq10$, we extend this to $n\leq15$. Of course here we
need to assume that Faber's perfect pairing conjecture continues to hold
in higher genera.
$$
\begin{array}{c||c|c|c|c|c|c|c|c|c|c|c|c|c|c|c}
 n & 1 & 2 &3 &4&5&6&7&8&9&10&11&12&13&14&15
\\ \hline a(n) & 1&1&2&3&5&6&10&13&18&24&33&41&56&71&91
\end{array}
$$

Motivated by Faber's calculation, we extracted two integer sequences
from the mock theta function $\omega(q)$, which may be of
independent interest.
\begin{proposition}
There exist two integer sequences $p_{\omega}(n), a_{\omega}(n)$
with $a_{\omega}(n)=0$ for $n\leq0$ such that for all $g\geq2$, if
we define
$$
\omega^k_g:=\begin{cases} p_{\omega}(k)-a_{\omega}(3k-g)), & 0\leq k\leq \frac{g-2}{2},\\
\omega^{g-2-k}_g, & \frac{g-2}{2}< k\leq g-2,
\end{cases}
$$
then $\omega_g=\sum_{k=0}^{g-2}\omega^k_g$.
 Moreover, $p_{\omega}(n)$ and $a_{\omega}(n)$ are uniquely
determined.
\end{proposition}
\begin{proof}
Consider $\omega_{2m}^{m-1}$ and $\omega_{2m+1}^{m-1}$; then we have
\begin{align*}
p_\omega(m-1)-a_\omega(m-3)=\omega_{2m}^{m-1}&=\omega_{2m}-2\sum_{i=0}^{m-2}\omega_{2m}^i\\
&=\omega_{2m}-2\sum_{i=0}^{m-2}(p_\omega(i)-a_\omega(3i-2m)),
\end{align*}
\begin{align*}
p_\omega(m-1)-a_\omega(m-4)=\omega_{2m+1}^{m-1}&=\frac{\omega_{2m+1}-2\sum_{i=0}^{m-2}\omega_{2m+1}^i}{2}\\
&=\frac{\omega_{2m+1}}{2}-\sum_{i=0}^{m-2}(p_\omega(i)-a_\omega(3i-2m-1)).
\end{align*}
It is not difficult to see that $p_\omega$ and $a_\omega$ are
uniquely determined recursively by the above two identities. The
integrality of $\omega_g^k, p_\omega, a_\omega$ is guaranteed by
Lemma \ref{even} that $\omega_{2m+1}$ is even.
\end{proof}

Let $p(n)$ denote the partition number of $n$. We may compare
$p_{\omega}(n)$ and $p(n)$ in the following tables.
$$
\begin{array}{c||c|c|c|c|c|c|c|c|c|c|c|c|c|c|c|c|c|c}
    n                               &0&1&2&3&4&5&6 &7 &8 &9 &10&11&12&13 &14 &15 &16 &17
\\\hline           p(n)             &1&1&2&3&5&7&11&15&22&30&42&56&77&101&135&176&231&297
\\\hline       p_{\omega}(n)        &1&1&2&3&5&7&11&15&22&30&41&56&75&100&132&172&225&289

\\\hline\hline a_{\omega}(n) &0&1&1&2&3&5&7 &10&13&18&25&34&44&58 &74 &97 &125&160
\end{array}
$$
Comparing the first few values of $a_{\omega}(n)$ and $a(n)$ leads us
to guess that $a_{\omega}(n)\geq a(n)$ may always hold.

$$\begin{array}{c|c|c|c|c|c}
 n & 50 & 100 & 300 & 500 & 800\\
\hline\hline p_{\omega}(n) & 191817& 176074114&
\approx 8\times 10^{15}&\approx 2\times 10^{21}&\approx 5\times 10^{27}\\
\hline p_{\omega}(n)/p(n) & 0.9392 & 0.9239 &0.9074 & 0.9021 & 0.8983\\
\end{array}
$$

\begin{table}[!htp]
\centering
\renewcommand{\baselinestretch}{1}\selectfont
\caption{Rank of Faber's intersection matrix}
\begin{tabular}{c||c|c|c|c|c|c|c|c|c|c|cc}\label{fabermatrix}
{$g$} & $R^0_g$ & $R^1_g$ & $R^2_g$& $R^3_g$ & $R^4_g$ &  $R^5_g$ & $R^6_g$ & $R^7_g$ & $R^8_g$ & $R^9_g$& $R^{10}_g$\\
\hline\hline {2}  & 1& & & & & & & & & & & \\
\hline {3}  & 1 & 1 & & & & & & & & & & \\
\hline  {4} & 1 & 1 & 1 & & & & & & & & & \\
\hline {5} & 1 & 1 & 1 & 1 & & & & & & & & \\
\hline {6}  & 1 & 1 & 2&1 &1 & & & & & & & \\
\hline {7}  & 1 & 1 &2 &2 &1 &1 & & & & & & \\
\hline {8}  & 1 & 1 & 2&2&2 &1 &1 & & & & & \\
\hline {9}  & 1 & 1 & 2&3 &3 &2 &1 &1 & & & & \\
\hline {10}  & 1 & 1 & 2& 3&4 &3 &2 & 1&1 & & & \\
\hline {11}  & 1 & 1 &2 &3 &4 &4 &3 &2 &1 &1 & & \\
\hline {12}  & 1 & 1 &2 &3 &5 &5 &5 &3 &2 &1 &1 & \\
\hline {13}  & 1 & 1 &2 & 3&5 & 6&6 &5 &3 &2 &1 & \\
\hline {14}  & 1 & 1 &2& 3& 5& 6&8 &6 &5 &3 &2 & \\
\hline {15}  & 1 & 1 &2 & 3&5 & 7&9 &9 &7 &5 &3 & \\
\hline {16}  & 1 & 1 &2 & 3& 5& 7&10 &10 &10 &7 &5 & \\
\hline {17}  & 1 & 1 &2 & 3& 5& 7&10 &12 &12 &10 &7 & \\
\hline {18}  & 1 & 1 &2 & 3& 5& 7& 11&13 &{\bf 16} &13 &11 & \\
\hline {19}  & 1 & 1 &2 & 3& 5& 7& 11&14 &17 &17 &14 & \\
\hline {20}  & 1 & 1 &2 & 3& 5& 7&11 &14 &19 &20 &19 & \\
\hline {21}  & 1 & 1 &2 & 3& 5& 7&11 & 15&20 &{\bf 24} &{\bf 24} & \\
\hline {22}  & 1 & 1 &2 & 3& 5& 7& 11& 15& 21& 25& {\bf 29}& \\
\hline {23}  & 1 & 1 &2 & 3& 5& 7& 11& 15& 21& 27& {\bf 32}& \\
\end{tabular}
\end{table}

\begin{table}[!htp]
\centering
\renewcommand{\baselinestretch}{1}\selectfont
\caption{Decomposition of $\omega_g$}
\begin{tabular}{c||c|c|c|c|c|c|c|c|c|c|cc}\label{omegamatrix}
{$g$} & $\omega^0_{g}$ & $\omega^1_{g}$ & $\omega^2_{g}$& $\omega^3_{g}$ & $\omega^4_{g}$ &  $\omega^5_{g}$ & $\omega^6_{g}$ & $\omega^7_{g}$ & $\omega^8_{g}$ & $\omega^9_{g}$& $\omega^{10}_{g}$\\

\hline {18}  & 1 & 1 &2 & 3& 5& 7& 11&13 &{\bf 15} &13 &11 & \\
\hline {19}  & 1 & 1 &2 & 3& 5& 7& 11&14 &17 &17 &14 & \\
\hline {20}  & 1 & 1 &2 & 3& 5& 7&11 &14 &19 &20 &19 & \\
\hline {21}  & 1 & 1 &2 & 3& 5& 7&11 & 15&20 &{\bf 23} &{\bf 23} & \\
\hline {22}  & 1 & 1 &2 & 3& 5& 7& 11& 15& 21& 25& {\bf 28}& \\
\hline {23}  & 1 & 1 &2 & 3& 5& 7& 11& 15& 21& 27& {\bf 31}& \\
\end{tabular}
\end{table}

\begin{conjecture} \label{dimconj} For all $n\geq 0$, we have
\begin{equation}\label{dimineq2}
p(n)\geq p_{\omega}(n),\qquad a(n)\leq a_{\omega}(n)
\end{equation}
 and there
exists some constant $C>0$, such that
$$\lim_{n\rightarrow\infty}\frac{p_{\omega}(n)}{p(n)}=C.$$
\end{conjecture}
Note that the two inequalities \eqref{dimineq2} together imply
\eqref{dimineq}. We may also strengthen the inequality
\eqref{dimineq} in Conjecture \ref{dimen} as follows.
\begin{conjecture} Let $g\geq2$ and $0\leq k\leq g-2$. Then we have
$$R_g^k\geq \omega^k_g.$$
\end{conjecture}

We have checked Conjecture \ref{dimconj} when $g\leq36$. Values of
$R_g^k$ and $\omega^k_g$ for $g\leq 23$ are listed in Table
\ref{fabermatrix} and Table \ref{omegamatrix} respectively; note
that $R_g^k=\omega^k_g$ when $g\leq17$. We use bold numbers whenever
the corresponding values are different.

 \vskip 30pt
\section{Gromov-Witten invariants}

Gromov-Witten theory is a generalization of the intersection theory
of moduli spaces of curves. In fact, the intersection theory of
$\overline{\mathcal M}_{g,n}$ corresponds to the Gromov-Witten theory
of a point. A very readable exposition of Gromov-Witten invariants
can be found in \cite{Ge}.

Let $X$ be a smooth projective variety and $\overline{\mathcal
M}_{g,n}(X,\beta)$ denote the moduli stack of stable maps of genus
$g$ and degree $\beta\in H_2(X, \mathbb Z)$ with $n$ marked points.
There are several canonical morphisms:
\begin{enumerate}
\item[i)] Let ${ev}$ be the evaluation maps at the
marked points:
\begin{align}\label{eval}
{ev}: \overline{\mathcal M}_{g,n}(X,\beta)&\rightarrow
X^n\\\nonumber
 {(f:C\to X, x_1,\dots,x_n) } &\mapsto \bigl(
f(x_1),\dots,f(x_n) \bigr) \in X^n .
\end{align}

\item[ii)] Let $st$ be the forgetful map to the domain curve followed
by stabilization:
\begin{equation}\label{st}
st : \overline{\mathcal M}_{g,n+1}(X,\beta) \rightarrow
\overline{\mathcal M}_{g,n}.
\end{equation}

\item[iii)] Let
$\pi$ be the map of forgetting the last marked point $x_{n+1}$ and
stabilizing the resulting domain curve:
\begin{equation}
\pi : \overline{\mathcal M}_{g,n+1}(X,\beta) \rightarrow
\overline{\mathcal M}_{g,n}(X,\beta)
\end{equation}
\end{enumerate}

The forgetful morphism $\pi$ has $n$ canonical sections
$$
\sigma_i : \overline{\mathcal M}_{g,n}(X,\beta)\rightarrow
\overline{\mathcal M}_{g,n+1}(X,\beta) ,
$$
corresponding to the $n$ marked points. Let
$$
\omega=\omega_{\overline{\mathcal
M}_{g,n+1}(X,\beta)/\overline{\mathcal M}_{g,n}(X,\beta)}
$$
be the relative dualizing sheaf and $\Psi_i$ the cohomology class
$c_1(\sigma_i^*\omega)$.

If $\gamma_1,\dots,\gamma_n\in H^*(X,\mathbb Q)$, the Gromov-Witten
invariants are defined by
$$
\langle\tau_{d_1}(\gamma_1) \dots
\tau_{d_n}(\gamma_n)\rangle^X_{g,\beta} = \int_{[\overline{\mathcal
M}_{g,n}(X,\beta)]^{\rm virt}} \Psi_1^{d_1} \cdots \Psi_n^{d_n} \cup
{\rm ev}^*(\gamma_1\boxtimes \cdots \boxtimes \gamma_n).
$$
We also denote the insertion $\tau_{k}(\gamma_a)$ by $\tau_{k,a}$.

Given a basis $\{\gamma_a\}$ for $H^*(X,\mathbb Q)$, we may use
$g_{ab}=\int_X \gamma_a\cup \gamma_b$ and its inverse $g^{ab}$ to
lower and raise indices. We denote $\gamma^a=g^{ab}\gamma_b$ and
apply the Einstein summation convention.

The genus $g$ Gromov-Witten potential of $X$ is defined by
$$\langle\langle\tau_{d_1}(\gamma_1)\cdots\tau_{d_n}(\gamma_n)\rangle\rangle_g=
\sum_{\beta}\left\langle\tau_{d_1}(\gamma_1)\cdots\tau_{d_n}(\gamma_n)\exp\left(\sum_{m,a}t_m^a\tau_m(\gamma_a)\right)\right\rangle^X_{g,\beta}q^\beta.$$

\subsection{Universal equations of Gromov-Witten invariants}
There are some universal equations satisfied by Gromov-Witten
invariants. Universal means that they do not depend on the target
manifolds.

We may pull back tautological relations on $\overline{\sM}_{g,n}$
via the map $st$ in \eqref{st} to get universal equations for
Gromov-Witten invariants by the splitting axiom and cotangent line
comparison equations \cite{KM}.

From the simple fact that the three boundary divisors of
$\overline{\sM}_{0,4}\cong\mathbb P^1$ are equal, we get the WDVV
equation:
\begin{equation}
\langle\langle\tau_{k_1,a_1}\tau_{k_2,a_2}\gamma_{\alpha}\rangle\rangle_0
\langle\langle\gamma^{\alpha}\tau_{k_3,a_3}\tau_{k_4,a_4}\rangle\rangle_0=
\langle\langle\tau_{k_1,a_1}\tau_{k_3,a_3}\gamma_{\alpha}\rangle\rangle_0
\langle\langle\gamma^{\alpha}\tau_{k_2,a_2}\tau_{k_4,a_4}\rangle\rangle_0,
\end{equation}
which is the associativity condition of the quantum cohomology ring.

On $\overline{\sM}_{1,1}$, we have $\psi_1=\frac{1}{12}\delta$,
which implies the genus one topological recursion relation
$$\langle\langle\tau_k(x)\rangle\rangle_1 =
\langle\langle\tau_{k-1}(x)\gamma_\alpha\rangle\rangle_0
\langle\langle\gamma^\alpha\rangle\rangle_1 + \frac{1}{24}
 \langle\langle\tau_{k-1}(x)\gamma_\alpha\gamma^\alpha\rangle\rangle_0
.$$ Other known topological recursion relations in genus $g\leq3$
are given in \cite{ArL, BP, Ge1, Ge2, KiX}.

Xiaobo Liu introduced the ``$T$ operator'' to facilitate the
transformation from topological relations on moduli spaces of curves
to universal equations of Gromov-Witten invariants by the splitting
axiom. The interested
reader may consult \cite{Liu2} for a discussion on relations among
known universal equations.

Using the WDVV equation, Kontsevich derived a recursion formula for
the number $N_d$ of degree $d$ plane rational curves passing through
$3d-1$ general points, illustrating the power of Gromov-Witten
theory in classical enumerative geomtry.

Inspired by Givental's axiomatic Gromov-Witten theory, Y.-P. Lee \cite{Lee1, Lee2} had proposed an algorithm that, conjecturally, computes all
tautological equations on $\overline\sM_{g,n}$ using only linear algebra. Faber, Shadrin and Zvonkine \cite{FSZ}
proved that Y.-P. Lee's algorithm is correct if and only if
the Gorenstein conjecture on the tautological cohomology ring of $\overline\sM_{g,n}$
is true.

There are also universal equations which do not come from the
tautological ring of moduli space of curves. For example, we have
the so-called string equation, dilaton equation and divisor equation,
respectively, in the following.
\begin{align}
\langle\tau_{0,0}\tau_{k_1,a_1}\cdots\tau_{k_n,a_n}\rangle_{g,\beta}^X
=& \sum_{i=1}^n \langle\tau_{k_1,a_1}\cdots\tau_{k_i-1,a_i} \cdots
\tau_{k_n,a_n}\rangle_{g,\beta}^X\\
\langle\tau_{1,0}\tau_{k_1,a_1}\cdots\tau_{k_n,a_n}\rangle_{g,\beta}^X
=& (2g-2+n)
\langle\tau_{k_1,\alpha_1}\cdots\tau_{k_n,\alpha_n}\rangle_{g,\beta}^X\\
\langle\tau_0(\omega)\tau_{k_1,a_1}\cdots\tau_{k_n,a_n}\rangle_{g,\beta}^X
=& \bigl( \omega
\cap \beta \bigr) \* \langle\tau_{k_1,a_1}\cdots\tau_{k_n,a_n}\rangle_{g,\beta}^X \\
\notag  &+ \sum_{i=1}^n
\langle\tau_{k_1,a_1}\cdots\tau_{k_i-1}(\omega\cup\gamma_a) \cdots
\tau_{k_n,a_n}\rangle_{g,\beta}^X,
\end{align}
where $\omega\in H^2(X,\mathbb Q)$.

\subsection{Some vanishing identities} We adopt Gathmann's convention \cite{Gath} in the following
to simplify notation, namely we define
$$\langle\tau_{-2}(pt)\rangle_{0,0}^X=1,$$
$$\langle\tau_{m}(\gamma_1)\tau_{-1-m}(\gamma_2)\rangle_{0,0}^X=(-1)^{\max(m,-1-m)}\int_X\gamma_1\cdot\gamma_2,
\quad m\in\mathbb Z.$$ All other Gromov-Witten invariants that
contain a negative power of a cotangent line are defined to be zero.

Motivated by our work on intersection numbers on moduli spaces of
curves, we \cite{LX6} conjectured the following universal equations for
Gromov-Witten invariants valid in all genera.

\begin{conjecture}\cite{LX6}\label{lx1}
Let $x_i, y_i\in H^*(X)$ and $k\geq 2g-3+r+s$. Then
$$\sum_{g'=0}^g\sum_{j\in\mathbb
Z}(-1)^j\langle\langle\tau_j(\gamma_a)\prod_{i=1}^r\tau_{p_i}(x_i)\rangle\rangle_{g'}
\langle\langle\tau_{k-j}(\gamma^a)\prod_{i=1}^s\tau_{q_i}(y_i)\rangle\rangle_{g-g'}=0.$$
Note that $j$ runs over all integers.
\end{conjecture}

\begin{conjecture}\cite{LX6}\label{lx2}
Let $k>g$. Then
\begin{equation}
\sum_{j=0}^{2k}(-1)^j\langle\langle\tau_{j}(T_a)\tau_{2k-j}(T^a)\rangle\rangle_g^X=0.
\end{equation}
We also have
\begin{equation}
\frac12\sum_{j=0}^{2g-2}(-1)^j\langle\langle\tau_{j}(T_a)\tau_{2g-2-j}(T^a)\rangle\rangle_{g-1}=\frac{(2g)!}{B_{2g}}\langle\langle{\rm
ch}_{2g-1}(\mathbb E)\rangle\rangle_g.
\end{equation}
\end{conjecture}

Note that by the Chern character formula of Faber and Pandharipande
\cite{FP1}, we have the equivalence\smallskip

\centerline{Conjecture \ref{lx1} ($r=s=0$) $\Longleftrightarrow$
Conjecture \ref{lx2}} \medskip

When $X$ is a point, the above conjectures have been proved in
\cite{LX6} using the recursive formula of $n$-point functions.
Recently, X. Liu and Pandharipande \cite{Liu,LiP} give a complete
proof of the above conjectures. Their proof uses virtual
localization to get topological recursion relations (TRR) expressing
$\psi_1^{2g+r}$ in terms of boundary classes in
$A^{2g+r}(\overline{\mathcal M}_{g,1})$, which are then pulled
back to get the universal equations of Conjecture \ref{lx1}.

As we have seen, it is relatively straightforward to go from TRR to
universal equations for Gromov-Witten invariants. This is not always
easy when an identity of descendent integrals is not a TRR. An
example is the Virasoro conjecture \cite{EHX} for Gromov-Witten
invariants, which is a generalization of the DVV formula \eqref{dvv}
and was not discovered until 6 years later. In the same line, one
would hope to find matrix models or corresponding integrable
hierarchies for a general target $X$. Besides the point case, we
know that the Gromov-Witten potential of $X=\mathbb{CP}^1$ is
governed by the Toda hierarchy \cite{EY, OP2}.

\vskip 30pt
\section{Witten's $r$-spin numbers}
In this section, we present an algorithm for computing Witten's
$r$-spin intersection numbers. First we recall Witten's definition
of $r$-spin numbers \cite{Wi2}.

Let $\Sigma$ be a Riemann surface of genus $g$ with marked points
$x_1,x_2,\dots,x_s$. Fix an integer $r\geq 2$. Label each marked
point $x_i$ by an integer $m_i$, $0\leq m_i\leq r-1$. Consider the
line bundle $\mathcal S=K\otimes(\otimes_{i=1}^s\mathcal
O(x_i)^{-m_i})$ over $\Sigma$, where $K$ denotes the canonical line
bundle. If $2g-2-\sum_{i=1}^s m_i$ is divisible by $r$, then there
are $r^{2g}$ isomorphism classes of line bundles $\mathcal T$ such
that $\mathcal T^{\otimes r}\cong \mathcal S$. The choice of an
isomorphism class of $\mathcal T$ determines a cover $\mathcal
M^{1/r}_{g,s}$ of $\mathcal M_{g,s}$. The compactification of
$\mathcal M^{1/r}_{g,s}$ is denoted by $\overline{\mathcal
M}^{1/r}_{g,s}$.

Let $\mathcal V$ be a vector bundle over $\overline{\mathcal
M}^{1/r}_{g,s}$ whose fiber is the dual space to
$H^1(\Sigma,\mathcal T)$. The top Chern class $c_D(\mathcal{V})$ of
this bundle has degree $ D=(g-1)(r-2)/r+\sum_{i=1}^sm_i/r. $

We associate with each marked point $x_i$ an integer $n_i\geq 0$.
Witten's $r$-spin intersection numbers are defined by
$$
\langle\tau_{n_1,m_1}\dots\tau_{n_s,m_s}\rangle_g=
\frac{1}{r^g}\int_{\overline{\mathcal
M}^{1/r}_{g,s}}\prod_{i=1}^s\psi(x_i)^{n_i}\cdot
 c_D(\mathcal{V}),
$$
which is non-zero only if
\begin{equation}
(r+1)(2g-2)+rs=r\sum_{j=1}^s n_j+\sum_{j=1}^s m_j.
\end{equation}

Consider the differential operator
$$
Q=D^r+\sum_{i=0}^{r-2}\gamma_i(x)D^i, \quad \text{ where }
D=\frac{\sqrt{-1}}{\sqrt{r}}\frac{\partial}{\partial x}.
$$

There exists a unique pseudo-differential operator
$Q^{1/r}=D+\sum_{i>0}w_{-i}D^{-i}$.

The Gelfand--Dikii equations read
\begin{equation}
i\frac{\partial Q}{\partial t_{n,m}} = [Q^{n+(m+1)/r}_+,Q]\cdot
\frac{c_{n,m}}{\sqrt{r}},
\end{equation}
where
$$
c_{n,m}=\frac{(-1)^{n}r^{n+1}} {(m+1)(r+m+1)\cdots(nr+m+1)}.
$$

\subsection{Generalized Witten's conjecture}
Consider the formal series $F$ in variables $t_{n,m}$, $n\geq 0$ and
$0\leq m\leq r-1$,
$$
F(t_{0,0},t_{0,1},\dots)=\sum_{d_{n,m}}\langle\prod_{n,m}\tau_{n,m}^{d_{n,m}}\rangle
\prod_{n,m}\frac{t_{n,m}^{d_{n,m}}}{d_{n,m}!}.
$$

The conjecture of Witten \label{Wi2} is that this $F$ is the string
solution of the $r$-Gelfand--Dikii hierarchy, namely
\begin{equation}
\frac{\partial F}{\partial t_{0,0}}=\frac{1}{2}
\sum_{i,j=0}^{r-2}\delta_{i+j,r-2}t_{0,i}t_{0,j}+
\sum_{n=0}^{\infty} \sum_{m=0}^{r-2} t_{n+1,m} \frac{\partial
F}{\partial t_{n,m}},
\end{equation}

\begin{equation}
\frac{\partial ^2 F}{\partial t_{0,0} \partial t_{n,m}}
=-c_{n,m}\mathrm{res} (Q^{n+\frac{m+1}{r}}),
\end{equation}
where $Q$ satisfies the Gelfand-Dikii equations and $t_{0,0}$ is
identified with $x$.

When $r=2$, the above assertion is just the Witten-Kontsevich
theorem. Witten's $r$-spin conjecture for any $r\geq 2$ has been
proved recently by Faber, Shadrin and Zvonkine \cite{FSZ}.
In fact,
Witten's $r$-spin theory corresponds to $A_{r-1}$ singularity. Fan,
Javis and Ruan \cite{FJR} have developed a Gromov-Witten type
quantum theory for all non-degenerate quasi-homogeneous singularity
and proved the more general ADE-integrable hierarchy conjecture of Witten.

\subsection{An algorithm for computing Witten's $r$-spin numbers} We proved a structure
theorem about formal pseudo-differential operators in a forthcoming paper \cite{LX9}. If combined with
the generalized Witten conjecture, it can be used to derive the following effective recursion formulae for computing
Witten's $r$-spin numbers.

\begin{theorem} \cite{LX9} For fixed $r\geq 2$, we have
$$
\langle\langle\tau_{1,0}\tau_{0,0}\rangle\rangle_g=\frac{1}{2}\langle\langle\tau_{0,0}\tau_{0,m'}\rangle\rangle_{g'}
\eta^{m'm''}\langle\langle\tau_{0,m''}\tau_{0,0}\rangle\rangle_{g-g'}
+{\rm Low}(r),
$$
where ${\rm Low}(r)$ is an explicit sum of products of
$\langle\langle\cdots\rangle\rangle$ with genera strictly lower than
$g$.
\end{theorem}

In particular, when $r=2$, we have
$$
\langle\langle\tau_{1,0}\tau_{0,0}\rangle\rangle_g=\frac{1}{2}\langle\langle\tau_{0,0}^2\rangle\rangle_{g'}\langle\langle\tau_{0,0}^2\rangle\rangle_{g-g'}
+\frac{1}{12}\langle\langle\tau_{0,0}^4\rangle\rangle_{g-1},
$$
when $r=3$, we have
$$
\langle\langle\tau_{1,0}\tau_{0,0}\rangle\rangle_g=\langle\langle\tau_{0,0}\tau_{0,1}\rangle\rangle_{g'}\langle\langle\tau_{0,0}^2\rangle\rangle_{g-g'}
+\frac{1}{6}\langle\langle\tau_{0,0}^3\tau_{0,1}\rangle\rangle_{g-1},
$$
when $r=4$, we have
\begin{multline*}
\langle\langle\tau_{1,0}\tau_{0,0}\rangle\rangle_g=\langle\langle\tau_{0,0}\tau_{0,2}\rangle\rangle_{g'}\langle\langle\tau_{0,0}^2\rangle\rangle_{g-g'}
+\frac{1}{2}\langle\langle\tau_{0,0}\tau_{0,1}\rangle\rangle_{g'}\langle\langle\tau_{0,1}\tau_{0,0}\rangle\rangle_{g-g'}\\
+\frac{1}{4}\langle\langle\tau_{0,0}^3\tau_{0,2}\rangle\rangle_{g-1}+\frac{1}{48}\langle\langle\tau_{0,0}^2\rangle\rangle_{g'}\langle\langle\tau_{0,0}^4\rangle\rangle_{g-1-g'}
+\frac{1}{32}\langle\langle\tau_{0,0}^3\rangle\rangle_{g'}\langle\langle\tau_{0,0}^3\rangle\rangle_{g-1-g'}\\
+\frac{1}{480}\langle\langle\tau_{0,0}^6\rangle\rangle_{g-2}.
\end{multline*}

Now we describe how to use the above Theorem to compute intersection
numbers. It consists of three steps.

\begin{enumerate}
\item[i)]
When $g=0$, these intersection numbers can be computed by the WDVV
equations. An algorithm has been given by Witten \cite{Wi2}.

\item[ii)] Let $g\geq1$.
For an intersection number containing a puncture operator
$\langle\tau_{0,0}\tau_{n_1,m_1}\cdots\tau_{n_s,m_s}\rangle_g$, we
have from Theorem 1.1 and the dilaton equation
\begin{multline*}
(2g-1+s-a)\langle\tau_{0,0}\tau_{n_1,m_1}\cdots\tau_{n_s,m_s}\rangle_g\\=\sum^{\sim
}_{\underline{s}=I\coprod J}\langle\tau_{0,0}\tau_{0,m'}\prod_{i\in
I}\tau_{n_i,m_i}\rangle_{g'}\eta^{m'm''}\langle\tau_{0,m''}\tau_{0,0}\prod_{i\in
J}\tau_{n_i,m_i}\rangle_{g-g'}+{\rm Low}(r)
\end{multline*}
where $a=\#\{i\mid n_i=0\}$. Note that in the summation
$$\sum^{\sim}_{\underline{s}=I\coprod J}$$
we rule out the cases $I=\{i_1\}$ and $n_{i_1}=0$ or $J=\{i_1\}$ and
$n_{i_1}=0$. Then the right-hand side follows by induction on genera
or number of marked points.

\item[iii)]
For any intersection number
$\langle\tau_{n_1,m_1}\cdots\tau_{n_s,m_s}\rangle_g$ with $n_1\geq
n_2\geq\cdots\geq n_s$, we apply the string equation first
\begin{multline*}
\langle\tau_{n_1,m_1}\cdots\tau_{n_s,m_s}\rangle_g=\langle\tau_{0,0}\tau_{n_1+1,m_1}\cdots\tau_{n_s,m_s}\rangle_g
-\sum_{j=2}^s\langle\tau_{n_1+1,m_1}\tau_{n_j-1,m_j}\prod_{i\neq 1,j}\tau_{n_i,m_i}\rangle_g\\
\end{multline*}
The first term on the right-hand side follows from step (ii) and the
second term follows by induction on the maximum descendent index.
\end{enumerate}

\

We have written a Maple program according to the above algorithm.
Some $r$-spin numbers when $r=3$ and $4$ are listed below. These
values agree with previous results in \cite{BH2, KiX, Sh}.

\begin{table}[!htp]
\centering
\caption{Witten's  $r$-spin numbers ($r=3$)}
\begin{tabular}{|c|c||c|c||c|c|}

\hline $\langle\tau_{1,0}\rangle_1$ & $\frac{1}{12}$&
$\langle\tau_{1,1}\tau_{3,1}\rangle_2$ & $\frac{11}{4320}$ &
 $\langle\tau_{0,1}\tau_{0,1}\tau_{2,1}\rangle_1$ & $\frac{1}{36}$
\\ \hline $\langle\tau_{6,1}\rangle_3$ & $\frac{1}{31104}$ & $\langle\tau_{2,1}\tau_{2,1}\rangle_2$ & $\frac{17}{4320}$ &
$\langle\tau_{0,1}\tau_{1,1}\tau_{1,1}\rangle_1$ & $\frac{1}{36}$
\\ \hline $\langle\tau_{9,0}\rangle_4$ & $\frac{1}{746496}$ & $\langle\tau_{0,1}\tau_{7,0}\rangle_3$ & $\frac{1}{15552}$
& $\langle\tau_{0,1}\tau_{0,1}\tau_{5,0}\rangle_2$ & $\frac{1}{432}$
\\ \hline $\langle\tau_{14,1}\rangle_6$ & $\frac{1}{4837294080}$ & $\langle\tau_{1,1}\tau_{6,0}\rangle_3$ & $\frac{19}{77760}$
& $\langle\tau_{0,1}\tau_{1,1}\tau_{4,0}\rangle_2$ & $\frac{13}{2160}$
\\ \hline $\langle\tau_{17,0}\rangle_7$ & $\frac{1}{162533081088}$ & $\langle\tau_{2,0}\tau_{5,1}\rangle_3$ & $\frac{103}{217728}$
& $\langle\tau_{0,1}\tau_{2,0}\tau_{3,1}\rangle_2$ & $\frac{1}{108}$
\\ \hline $\langle\tau_{22,1}\rangle_9$ & $\frac{1}{1805510340771840}$ & $\langle\tau_{2,1}\tau_{5,0}\rangle_3$ & $\frac{47}{77760}$
&$ \langle\tau_{0,1}\tau_{2,1}\tau_{3,0}\rangle_2$ & $\frac{23}{2160}$

\\ \hline  $\langle\tau_{25,0}\rangle_{10}$ & $\frac{1}{75831434312417280}$
&  $\langle\tau_{3,0}\tau_{4,1}\rangle_3$ & $\frac{443}{544320}$ &
$\langle\tau_{1,1}\tau_{1,1}\tau_{3,0}\rangle_2$ & $\frac{29}{2160}$
\\ \hline  $\langle\tau_{30,1}\rangle_{12}$ & $\frac{1}{1235489060066080849920}$
&  $\langle\tau_{3,1}\tau_{4,0}\rangle_3$ & $\frac{67}{77760}$ &
$\langle\tau_{1,1}\tau_{2,0}\tau_{2,1}\rangle_2$ & $\frac{19}{1080}$
\\\hline
\end{tabular}
\end{table}

\begin{table}[!htp]
\centering
\caption{Witten's  $r$-spin numbers ($r=4$)}
\begin{tabular}{|c|c||c|c||c|c|}

\hline $\langle\tau_{1,0}\rangle_1$ & $\frac{1}{8}$ &
$\langle\tau_{0,2}\tau_{1,2}\rangle_1$ & $\frac{1}{96}$ &
 $\langle\tau_{0,1}\tau_{0,1}\tau_{2,2}\rangle_1$ & $\frac{1}{32}$

\\ \hline $\langle\tau_{3,2}\rangle_2$ & $\frac{3}{2560}$ & $\langle\tau_{0,1}\tau_{4,1}\rangle_2$ & $\frac{1}{320}$ &
$\langle\tau_{0,1}\tau_{0,2}\tau_{2,1}\rangle_1$ & $\frac{1}{24}$

\\ \hline $\langle\tau_{6,0}\rangle_3$ & $\frac{3}{20480}$ & $\langle\tau_{0,2}\tau_{4,0}\rangle_2$ & $\frac{19}{7680}$
& $\langle\tau_{0,1}\tau_{1,1}\tau_{1,2}\rangle_1$ & $\frac{1}{24}$

\\ \hline $\langle\tau_{8,2}\rangle_4$ & $\frac{77}{39321600}$ & $\langle\tau_{1,1}\tau_{3,1}\rangle_2$ & $\frac{7}{960}$
& $\langle\tau_{0,2}\tau_{0,2}\tau_{2,0}\rangle_1$ & $\frac{1}{48}$

\\ \hline $\langle\tau_{11,0}\rangle_5$ & $\frac{19}{104857600}$ & $\langle\tau_{1,2}\tau_{3,0}\rangle_2$ & $\frac{41}{7680}$
& $\langle\tau_{0,1}\tau_{0,1}\tau_{5,0}\rangle_2$ & $\frac{13}{2560}$

\\ \hline $\langle\tau_{13,2}\rangle_6$ & $\frac{59}{33554432000}$ & $\langle\tau_{2,0}\tau_{2,2}\rangle_2$ & $\frac{49}{7680}$
& $\langle\tau_{0,1}\tau_{1,1}\tau_{4,0}\rangle_2$ & $\frac{1}{64}$

\\ \hline $\langle\tau_{16,0}\rangle_{7}$ & $\frac{39}{268435456000}$ &  $\langle\tau_{2,1}\tau_{2,1}\rangle_2$ & $\frac{11}{960}$
& $\langle\tau_{0,1}\tau_{2,1}\tau_{3,0}\rangle_2$ & $\frac{9}{320}$

\\ \hline $\langle\tau_{18,2}\rangle_{8}$ & $\frac{9367}{7215545057280000}$ &  $\langle\tau_{2,0}\tau_{5,0}\rangle_3$ & $\frac{43}{20480}$
& $\langle\tau_{1,1}\tau_{1,1}\tau_{3,0}\rangle_2$ & $\frac{7}{192}$

\\ \hline $\langle\tau_{21,0}\rangle_{9}$ & $\frac{2363}{24739011624960000}$ &  $\langle\tau_{3,0}\tau_{4,0}\rangle_3$ & $\frac{7}{2048}$
& $\langle\tau_{1,1}\tau_{2,0}\tau_{2,1}\rangle_2$ & $\frac{1}{20}$

\\ \hline $\langle\tau_{23,2}\rangle_{10}$ & $\frac{23567}{30786325577728000000}$ &  $\langle\tau_{1,2}\tau_{5,2}\rangle_3$ & $\frac{311}{1720320}$
& $\langle\tau_{0,2}\tau_{2,2}\tau_{2,2}\rangle_2$ & $\frac{11}{3072}$

\\ \hline $\langle\tau_{26,0}\rangle_{11}$ & $\frac{5443}{105553116266496000000}$ &  $\langle\tau_{2,2}\tau_{4,2}\rangle_3$ & $\frac{67}{172032}$
& $\langle\tau_{1,2}\tau_{1,2}\tau_{2,2}\rangle_2$ & $\frac{7}{1536}$
\\ \hline
\end{tabular}
\end{table}

$$ \ \ \ \ $$

\end{document}